\documentclass[preprint,1p]{elsarticle}

\usepackage[pagewise]{lineno}
\usepackage{mathrsfs}
\usepackage{latexsym,amsfonts,amssymb,amsmath,amsthm,color,bbm}
\usepackage{graphicx,subfigure}
\usepackage{amsbsy,amsfonts,amsmath,amssymb,bbm,calc,caption,color,dsfont,graphics}
\usepackage{graphicx,ifthen,mathptmx,mathrsfs}
\usepackage{algorithm}
\usepackage{amsmath}
\usepackage{amsfonts}
\usepackage{epsfig}
\usepackage{epstopdf}
\usepackage[colorlinks=true, linkcolor=blue]{hyperref}
\usepackage{geometry}                  
\begin{document}
	\setlength{\baselineskip}{16pt}

	\newtheorem{theorem}{Theorem}[section]
	\newtheorem{lemma}{Lemma}[section]
	\newtheorem{proposition}{Proposition}[section]
	\newtheorem{definition}{Definition}[section]
	\newtheorem{example}{Example}[section]
	\newtheorem{corollary}{Corollary}[section]
	\newtheorem{assumption}{Assumption}[section]
	\newtheorem{remark}{Remark}[section]
	\numberwithin{equation}{section}
	\renewcommand{\labelenumi}{(\arabic{enumi})}

	\def\disp{\displaystyle}
	\def\undertex#1{$\underline{\hbox{#1}}$}
	\def\card{\mathop{\hbox{card}}}
	\def\sgn{\mathop{\hbox{sgn}}}
	\def\exp{\mathop{\hbox{exp}}}
	\def\OFP{(\Omega,{\cal F},\PP)}
	\newcommand\JM{Mierczy\'nski}
	\newcommand\RR{\ensuremath{\mathbb{R}}}
	\newcommand\EE{\ensuremath{\mathbb{E}}}
	\newcommand\CC{\ensuremath{\mathbb{C}}}
	\newcommand\QQ{\ensuremath{\mathbb{Q}}}
	\newcommand\ZZ{\ensuremath{\mathbb{Z}}}
	\newcommand\NN{\ensuremath{\mathbb{N}}}
	\newcommand\PP{\ensuremath{\mathbb{P}}}
	\newcommand\abs[1]{\ensuremath{\lvert#1\rvert}}
	\newcommand\normf[1]{\ensuremath{\lVert#1\rVert_{f}}}
	\newcommand\normfRb[1]{\ensuremath{\lVert#1\rVert_{f,R_b}}}
	\newcommand\normfRbone[1]{\ensuremath{\lVert#1\rVert_{f, R_{b_1}}}}
	\newcommand\normfRbtwo[1]{\ensuremath{\lVert#1\rVert_{f,R_{b_2}}}}
	\newcommand\normtwo[1]{\ensuremath{\lVert#1\rVert_{2}}}
	\newcommand\norminfty[1]{\ensuremath{\lVert#1\rVert_{\infty}}}
	
	\newcommand\del[1]{}
	
\begin{frontmatter}
%
%

\title{
	Dynamics of discrete random Burgers-Huxley systems: attractor convergence and finite-dimensional approximations\tnoteref{Supported}}
		\tnotetext[Supported]{The research is supported by National Natural Science Foundation of China (12371198).}



\author{Fang Su}
		\author{Xue Wang\corref{cor1}}

		
		\cortext[cor1]{Corresponding author, wangxue22@nudt.edu.cn}
		
		\address{College of Sciences, National University of Defense Technology, Changsha Hunan, 410073, P.R.China}

\begin{abstract}

In this paper, we apply the implicit Euler scheme to discretize the (random) Burgers-Huxley equation and prove the existence of a numerical attractor for the discrete Burgers-Huxley system. We establish upper semi-convergence of the numerical attractor to the global attractor as the step size tends to zero. We also provide finite-dimensional approximations for the three attractors (global, numerical and random) and prove the existence of truncated attractors as the state space dimension goes to infinity. Finally, we prove the existence of a random attractor and establish that the truncated random attractor upper semi-converges to the truncated global attractor as the noise intensity tends to zero.

\end{abstract}


\end{frontmatter}





\section{Introduction}\label{intro}

The Burgers-Huxley equation serves as a foundational model that captures the interplay between nonlinear convection, reaction dynamics, and diffusion processes. It reveals several captivating phenomena, including bursting oscillations \cite{Duan-2006}, population genetics \cite{Aronson-1978}, interspike intervals \cite{Liu-2005}, bifurcations, and chaos \cite{Zhang-2006}. These phenomena have applications in membrane models influenced by potassium and sodium ion fluxes \cite{Lewis-1964}. Numerous studies have investigated both the theoretical and numerical aspects of the Burgers-Huxley equation \cite{Gupt-2011, Kabeto-2021, Liu-2020}. Nevertheless, as novel scientific and technological challenges emerge, re-evaluating deterministic equations within the framework of random factors has become crucial for a more effective understanding of phenomena vulnerable to external perturbations. 
In the context of stochastic Burgers-Huxley equations, pioneering work of Mohan has established fundamental results including  the asymptotic behavior of solutions and global solvability \cite{Mohan-2020,Mohan-2022}. Most recently, Wang et al. \cite{Wang-2024-1} studied the approximation of invariant measures for stochastic Burgers-Huxley equations driven by additive Gaussian noise.
While these advances provide critical insights into solution behaviors and statistical properties, a systematic analysis of the long-term dynamics presents a significant challenge, particularly through the view of attractor theory.

Attractors provide a powerful framework for characterizing the long-term dynamics of solutions to partial differential equations.  In recent years, significant research attention has been devoted to investigating attractors in both deterministic and stochastic systems.
For example, rough differential equations \cite{Duc-2023}, porous media partial differential equations \cite{Li-2023-1}, non-autonomous Benjamin-Bona-Mahony equations \cite{Chen}, $p$-Laplace lattice dynamical systems \cite{Li-2023-2}, stochastic non-autonomous Zakharov lattice equations \cite{Li-2024-1}, Kuramoto-Sivashinsky lattice equations \cite{Li-2025-1}, Korteweg-de Vries equations \cite{Liu-2024}, stochastic discrete plate equations \cite{Han-2024}, sine-Gordon lattice equations \cite{Yang-2022}, and Navier-Stokes equations \cite{Su-2024, Bortolan-2025, Liu-2025, Zhang-2024, Enlow-2024}. To the best of our knowledge, there has been no the existence and characterization of global attractors for either the deterministic or stochastic Burgers-Huxley equations, despite their fundamental role in describing long-term dynamical behaviors.

In this paper, we consider the following Burgers-Huxley equation with the initial condition given by:
\begin{equation}\label{1.1}
\left\{\begin{array}{l}
d u(t)=\left(\nu\frac{\partial^2u(t)}{\partial x^2}-\alpha u(t)\frac{\partial u(t)}{\partial x}+\beta u(t)(1-u(t))(u(t)-\gamma)-\lambda u(t)+f \right)d t+d W(t),\\
u(x,0)=u_0(x),
\end{array}\right.
\end{equation}
where $\alpha > 0$ is the advection coefficient, $\lambda > 0$ is the damping constant, and $\gamma\in \left(0,1\right)$,  $\nu,\beta>0$ are parameters.  $f=\left(f_{i}\right)_{i \in \mathbb{Z}} \in\ell^2$ represents the determined external force, and $W(t)$ is a Wiener process defined on a complete filtered probability space $\left(\Omega, \mathcal{F},\left\{\mathcal{F}_t\right\}_{t \in \mathbb{R}}, \mathbb{P}\right)$.

When $\alpha=0$, equation \eqref{6.1} is known as the random Huxley equation, which models nerve pulse propagation in nerve fibers and wall motion in liquid crystals, as discussed in \cite{Wang-1985}. When $\beta=0$ and $\alpha=1$, equation \eqref{6.1} simplifies to the classical random viscous Burgers equation. In \cite{Zou-2017}, the authors explored the existence and uniqueness of mild solutions for the time- and space-fractional random Burgers equation influenced by multiplicative white noise. Lv and Duan \cite{Lv-2017} examined the existence of martingale and weak solutions for the space-fractional random nonlocal Burgers equation on a bounded domain. Exponential ergodicity for random Burgers equations was demonstrated in \cite{Goldys-2005}.


The goal of this paper is to investigate the existence and finite-dimensional approximation of numerical attractors for the (random) Burgers-Huxley equations, where the implicit Euler scheme (IES) is used for time discretization. For the deterministic Burgers-Huxley equation, we first prove the existence of a unique solution to the IES and show that there  exists a numerical attractor $\mathcal{A}^{\epsilon}$ for the discrete-time system, which is upper semi-convergence to the global attractor $\mathcal{A}$ of the continuous-time lattice model as the time step size $\epsilon\rightarrow 0$, which can be found in Theorem \ref{theorem3.3}. Subsequently, we consider the finite-dimensional approximation of the numerical attractor $\mathcal{A}_m^{\epsilon}$ and prove the existence of a finite-dimensional numerical attractor $\mathcal{A}_m^{\epsilon}$ and its convergence to the numerical attractor $\mathcal{A}^{\epsilon}$ of the infinite-dimensional IES as $m\rightarrow\infty$, as established in Theorem \ref{theorem4.2}.  

For the random Burgers-Huxley equation, we establish the existence of a random attractor using the Ornstein-Uhlenbeck process and tail estimates of the random solution. Next, we investigate the upper semi-continuity between the random attractor $\mathcal{A}^{\epsilon}(\omega)$ and global attractor $\mathcal{A}$ as $\epsilon\rightarrow0$, which can be found in Theorem \ref{theorem6.3}. The truncated system of the IES for the random Burgers-Huxley equation has a corresponding truncated random attractor $\mathcal{A}_{m}^{\epsilon}(\omega)$, which converges to the random attractor $\mathcal{A}^{\epsilon}(\omega)$ as $m\rightarrow\infty$, which is introduced in Theorem \ref{theorem6.2}. Finally, based on the convergence of solutions between the deterministic and random equations in the truncation sense, we demonstrate the upper semi-continuity from $\mathcal{A}_{m}^{\epsilon}(\omega)$ to the truncated global attractor $\mathcal{A}_{m}$, Theorem \ref{theorem6.4} demonstrates it.

The paper is organized as follows. In Sect. 2, we prove the existence of a unique solution for the IES of deterministic Burgers-Huxley equation by applying the contraction mapping principle and introduce the error estimate between the continuous-time solution and the discrete-time solution. In Sect. 3, we prove the existence of the numerical attractor $\mathcal{A}^{\epsilon}$ for the IES and investigate the upper convergence of the numerical attractor $\mathcal{A}^{\epsilon}$ to the global attractor $\mathcal{A}$ under the Hausdorff semi-distance. In Sect.4, we consider finite-dimensional approximation of the discrete-time lattice system, show that the truncated attractor $\mathcal{A}_m^{\epsilon}$ is upper semi-convergence to the numerical attractor $\mathcal{A}^{\epsilon}$, and derive bounds and continuity for both the numerical attractor $\mathcal{A}^{\epsilon}$ and the truncated attractor $\mathcal{A}_m^{\epsilon}$. Similar results for random Burgers-Huxley equation are provided in Sect.5.

\section{Well-posedness and discretization error}\label{sec:2}

We start with the deterministic Burgers-Huxley eqaution:
\begin{equation}\label{2.0}
\left\{\begin{array}{l}
\frac{d u(t)}{d t}=\nu\frac{\partial^2u(t)}{\partial x^2}-\alpha u(t)\frac{\partial u(t)}{\partial x}+\beta u(t)(1-u(t))(u(t)-\gamma)-\lambda u(t)+f,\\
u(x,0)=u_0(x).
\end{array}\right.
\end{equation}

According to \cite{Gu-2016}, let $\ell^2$
be the Hilbert space of real-valued bi-infinite sequences that are square summable. The inner product on $\ell^2$ is defined as
$$
(u,v)=\sum_{i\in \mathbb{Z}}u_iv_i, ~~~\forall u=(u_i)_{i\in \mathbb{Z}},~ v=(v_i)_{i\in \mathbb{Z}} \in \ell^2,
$$
and the norm $\|u\|=\sqrt{(u, u)}$. $\ell^2$ is the set of all bi-infinite sequences $u=\left(u_i\right)_{i \in \mathbb{Z}}$ such that the sum of the squares of their elements is finite, i.e.,
\begin{align*}
\ell^2:=\left\{u=\left(u_i\right)_{i \in \mathbb{Z}}:\|u\|^2=\sum_{i \in \mathbb{Z}}\left|u_i\right|^2<\infty\right\}.
\end{align*}
In the following content, we will use the $\ell^p$-norm, which is defined similarly to the $\ell^2$-norm
\begin{align*}
\ell^p:=\left\{u=\left(u_i\right)_{i \in \mathbb{Z}}:\|u\|_p^{p}=\sum_{i \in \mathbb{Z}}\left|u_i\right|^p<\infty\right\}.
\end{align*}
We introduce the operators $\Lambda$, $D^{+}$, $D^{-}$, which are linear operators from $\ell^{2}$ to $\ell^{2}$. For any $i\in \mathbb{Z}$ and $u=\left(u_i\right)_{i \in \mathbb{Z}}\in \ell^2$, they are defined as follows:
\begin{equation}\label{x2.1}
(\Lambda u)_{i}=-u_{i-1}+2u_{i}-u_{i+1}, \quad\left(D^{+} u\right)_i:=u_{i+1}-u_i, \quad\left(D^{-} u\right)_i:=u_{i-1}-u_i.
\end{equation}
Then, we have $\left(\Lambda u, u\right)=$ $-\|D^{+} u\|^2$, $\Lambda=D^{+}D^{-}=D^{-}D^{+}$, and $(D^{-}u,v)=(u,D^{+}v)$ for each $u,v\in \ell^{2}$. According to \cite{Yang-2022}, all operators are bounded on $\ell^{2}$ with $\|\Lambda\|\leq 4$ and $\|D^{+}\|=\|D^{-}\|\leq 2$. 

The lattice system of the Burgers-Huxley equation \eqref{2.0} can be written in the following form
\begin{equation}\label{2.1}
\frac{d u(t)}{d t}=\nu\Lambda u-\alpha uD^{-}u+\beta u(1-u)(u-\gamma)-\lambda u+f:=Fu,
\end{equation}
where the operator $F$ is called the vector field of the Burgers-Huxley equation. In this article, we use the IES to obtain the discrete-time system of equation \eqref{2.1} as follows:
\begin{equation}\label{2.2}
 u_{n}^{\epsilon}=u_{n-1}^{\epsilon}+\epsilon Fu_{n}^{\epsilon},
\end{equation}
where $\epsilon > 0$ is the time step, and  $u_{n}^{\epsilon}=\left(u_{n,i}^{\epsilon}\right)_{i \in \mathbb{Z}}$.

We will prove the existence of a unique solution for the IES \eqref{2.2} by applying the contraction mapping principle. First, we introduce two Lipschitz constants for the vector field of the Burgers-Huxley equation.

\begin{lemma}\label{lemma2.1} Let the operator $F$ be defined in \eqref{2.1}. Then for any $u,v\in B_r$, the following equalities hold
\begin{equation}\label{2.3}
\begin{aligned}
\|Fu\|\leq \beta r^3+(2\alpha+\beta+\beta\gamma)r^2+(4\nu+\beta\gamma+\lambda)r+\|f\|=:M_r,\hspace{3.2cm}
\end{aligned}
\end{equation}
\begin{equation}\label{2.4}
\begin{aligned}
\|Fu-Fv\|\leq L_r\|u-v\|, \quad with ~ L_r:=4\nu + 2\sqrt{5}r\alpha + \sqrt{12r^2(1+\gamma)^2+27r^4+3\gamma^2}\beta + \lambda,
\end{aligned}
\end{equation}
where
\begin{equation*}
B_r:=\{u\in \ell^2; \|u\|:=(\sum_{i \in \mathbb{Z}}u_i^2)^{1/2}\leq r \},  \quad\forall r>0.
\end{equation*}
\end{lemma}
\begin{proof}
For any $r>0$, let $u\in B_r$, we obtain
\begin{equation*}
\begin{aligned}
\|Fu\|&=\|\nu\Lambda u-\alpha uD^{-}u+\beta u(1-u)(u-\gamma)-\lambda u+f\|\\
&\leq \nu\|\Lambda u\|+\alpha \|uD^{-}u\|+\beta \|u(1-u)(u-\gamma)\|+\lambda \|u\|+\|f\|\\
&\leq\beta \|u\|^3+(2\alpha+\beta+\beta\gamma)\|u\|^2+(4\nu+\beta\gamma+\lambda)\|u\|+\|f\|\\
&\leq\beta r^3+(2\alpha+\beta+\beta\gamma)r^2+(4\nu+\beta\gamma+\lambda)r+\|f\|=:M_r.
\end{aligned}
\end{equation*}
Suppose that $u,v\in B_r$, then we get
\begin{equation*}
\begin{aligned}\|uD^{-}u-vD^{-}v\|^2&=\sum_{i \in \mathbb{Z}}|u_i(u_{i-1}-u_i)-v_i(v_{i-1}-v_i)|^2\\
&=\sum_{i \in \mathbb{Z}}|(u_{i-1}-u_i-v_i)(u_i-v_i)+v_i(u_{i-1}-v_{i-1})|^2\\
&\leq2\times(3r)^2\sum_{i \in \mathbb{Z}}|u_i-v_i|^2+2r^2\sum_{i \in \mathbb{Z}}|u_{i-1}-v_{i-1}|^2\\ &=20r^2\|u-v\|^2,
\end{aligned}
\end{equation*}
here we utilize the inequality $|u_i|\leq\|u\|$ for any $i \in \mathbb{Z}$, and then
\begin{equation}\label{2.5}
\|uD^{-}u-vD^{-}v\|=2\sqrt{5}r\|u-v\|,~~~\forall u,v\in B_r.
\end{equation}
We can also conclude that
\begin{equation*}
\begin{aligned}
\|u(1-u)(u-\gamma)-v(1-v)(v-\gamma)\|^2&=\sum_{i \in \mathbb{Z}}|u_i(1-u_i)(u_i-\gamma)-v_i(1-v_i)(v_i-\gamma)|^2\\
&=\sum_{i \in \mathbb{Z}}|-u_i^3+(1+\gamma)u_i^2-\gamma u_i+v_i^3+(1+\gamma)v_i^2-\gamma v_i|^2\\
&=\sum_{i \in \mathbb{Z}}[(1+\gamma)(u_i+v_i)(u_i-v_i)-(u_i^2+u_i v_i+v_i^2)(u_i-v_i)-\gamma(u_i-v_i)]^2\\
&\leq\sum_{i \in \mathbb{Z}}3[(1+\gamma)^2(u_i+v_i)^2+(u_i^2+u_i v_i+v_i^2)^2+\gamma^2]|u_i-v_i|^2\\
&\leq[12r^2(1+\gamma)^2+27r^4+3\gamma^2]\|u-v\|^2,
\end{aligned}
\end{equation*}
and thus
\begin{equation}\label{2.6}
\begin{aligned}
\|u(1-u)(u-\gamma)-v(1-v)(v-\gamma)\|\leq \sqrt{12r^2(1+\gamma)^2+27r^4+3\gamma^2}\|u-v\|.
\end{aligned}
\end{equation}
Together with \eqref{2.5} and \eqref{2.6}, this implies that
\begin{equation*}
\begin{aligned}
\|Fu-Fv\|&=\|\nu\Lambda u - \nu\Lambda v - \alpha \left(uD^{-}u - vD^{-}v\right) +\beta \left[ u(1-u)(u-\gamma) - v(1-v)(v-\gamma) \right]-\lambda \left( u-v\right)\|\\
& \leq \nu\|\Lambda \left( u-v\right)\|+\alpha \|uD^{-}u - vD^{-}v\|+\beta \|u(1-u)(u-\gamma) - v(1-v)(v-\gamma)\|+\lambda \|u-v\|\\
& \leq \left( 4\nu + 2\sqrt{5}r\alpha + \sqrt{12r^2(1+\gamma)^2+27r^4+3\gamma^2}\beta + \lambda \right)\|u-v\|\\
& = L_r\|u-v\|.
\end{aligned}
\end{equation*}
The proof of the lemma is complete.
\end{proof}

\subsection{Existence of a unique solution to the IES of the Burgers-Huxley equation}
Define the constant $\lambda^{*}$ as
\begin{equation}\label{2.1-12}
\lambda^{*} = 4\nu+ \frac{\left(2\alpha+\beta+\beta\gamma\right)^{2}}{4\beta} - \beta\gamma,
\end{equation}  
and assume that
\begin{equation}\label{2.1-0}
\lambda > \lambda^{*} .
\end{equation}         
We choose the special radius and time-size by
\begin{equation}\label{2.1-11}
r^{*} = 1+\frac{\|f\|}{\lambda-\lambda^{*}},~~~~~\epsilon^{*} = \min \left( \frac{1}{M_{r^{*}+1}}, \frac{1}{1+L_{r^{*}+1}}\right),
\end{equation} 
here $M_{r}$ and $L_{r}$ are given by \eqref{2.3} and \eqref{2.4}.

\begin{theorem}\label{theorem2.1}    
For any $\epsilon \in ( 0 ,\epsilon^{*}]$ and $u_{0}\in \textit{B}_{r^{*}}$, the IES \eqref{2.2} for the Burgers-Huxley equation has a unique solution $u_{n}^{\epsilon}(u_0)\in \textit{B}_{r^{*}}$.
\end{theorem}
\begin{proof}
Firstly, we prove positive invariance, i.e., the solution $u^{\epsilon}_{n}$ of the IES \eqref{2.2} satisfies $u_{n}^{\epsilon}\in \textit{B}_{r^{*}}$ if $u_{n-1}^{\epsilon}\in \textit{B}_{r^{*}}$. We take the inner product of \eqref{2.2} with $u^{\epsilon}_{n}$ to obtain 
\begin{equation}\label{2.1-1}
\|u^{\epsilon}_{n}\|^{2}=\left( u^{\epsilon}_{n-1}, u^{\epsilon}_{n}\right)+\epsilon\nu\left(\Lambda u^{\epsilon}_{n}, u^{\epsilon}_{n}\right) - \epsilon\alpha \left( u^{\epsilon}_{n} D^{-} u^{\epsilon}_{n}, u^{\epsilon}_{n}\right) + \epsilon\beta \left( u^{\epsilon}_{n}(1-u^{\epsilon}_{n})(u^{\epsilon}_{n}-\gamma), u^{\epsilon}_{n}\right)-\epsilon\lambda\left( u^{\epsilon}_{n}, u^{\epsilon}_{n}\right)+\epsilon\left( f, u^{\epsilon}_{n}\right).
\end{equation}
In the following, we will repeatedly use Young's inequality. The first term on the right-hand side of the equality \eqref{2.1-1} is estimated as below:
\begin{equation}\label{2.1-2}
\left( u^{\epsilon}_{n-1}, u^{\epsilon}_{n}\right)\leq \frac{1}{2}\|u^{\epsilon}_{n-1}\|^2 + \frac{1}{2}\|u^{\epsilon}_{n}\|^2. 
\end{equation} 
The second term is estimated as below:
\begin{equation}\label{2.1-3}
\begin{aligned}
\epsilon\nu\left(\Lambda u^{\epsilon}_{n}, u^{\epsilon}_{n}\right) & = \epsilon\nu \sum_{i\in \mathbb{Z}} \left( - u_{n,i+1}^{\epsilon} + 2 u_{n,i}^{\epsilon}- u_{n,i-1}^{\epsilon}\right)u_{n,i}^{\epsilon}\\
& = \epsilon\nu \sum_{i\in \mathbb{Z}} \left[ - u_{n,i+1}^{\epsilon}u_{n,i}^{\epsilon} + 2 \left(u_{n,i}^{\epsilon}\right)^{2}- u_{n,i-1}^{\epsilon}u_{n,i}^{\epsilon}\right]\\
& \leq \epsilon\nu \sum_{i\in \mathbb{Z}} \left[ \frac{1}{2}\left(u_{n,i+1}^{\epsilon}\right)^{2} + \frac{1}{2}\left(u_{n,i}^{\epsilon}\right)^{2}+ 2 \left(u_{n,i}^{\epsilon}\right)^{2} + \frac{1}{2}\left(u_{n,i-1}^{\epsilon}\right)^{2} + \frac{1}{2}\left(u_{n,i}^{\epsilon}\right)^{2}\right]\\
& = 4\epsilon\nu \|u^{\epsilon}_{n}\|^{2}.
\end{aligned}
\end{equation} 
The third term is estimated as follows:
\begin{equation*}
\begin{aligned}
- \epsilon\alpha \left( u^{\epsilon}_{n} D^{-} u^{\epsilon}_{n}, u^{\epsilon}_{n}\right) & = - \epsilon\alpha  \sum_{i\in \mathbb{Z}} \left( u_{n,i}^{\epsilon}\right)^{2} \left( u_{n,i-1}^{\epsilon} - u_{n,i}^{\epsilon}\right)\\
& \leq \epsilon\alpha  \sum_{i\in \mathbb{Z}} \left( u_{n,i}^{\epsilon}\right)^{3} + \epsilon\alpha \sum_{i\in \mathbb{Z}} \left( \frac{2}{3}|u_{n,i}^{\epsilon}|^3 + \frac{1}{3} |u_{n,i-1}^{\epsilon}|^3\right)\\
& = 2\epsilon\alpha  \|u_{n}^{\epsilon}\|^{3}_3.
\end{aligned}
\end{equation*} 
The fourth term is rewritten as
\begin{equation*}
\begin{aligned}
\epsilon\beta \left( u^{\epsilon}_{n}(1-u^{\epsilon}_{n})(u^{\epsilon}_{n}-\gamma), u^{\epsilon}_{n}\right) & = \epsilon\beta \left( -\left(u^{\epsilon}_{n}\right)^{3}+(1+\gamma) \left( u^{\epsilon}_{n}\right)^{2} - \gamma u^{\epsilon}_{n}, u^{\epsilon}_{n} \right)\\
& = \epsilon\beta \sum_{i\in \mathbb{Z}} \left[ -\left(u^{\epsilon}_{n,i}\right)^{4}+(1+\gamma) \left( u^{\epsilon}_{n,i}\right)^{3} - \gamma \left(u^{\epsilon}_{n,i} \right)^{2} \right]\\
& = -\epsilon\beta \|u^{\epsilon}_{n}\|^{4}_{4}+ \epsilon\beta(1+\gamma)\|u^{\epsilon}_{n}\|^{3}_{3} - \epsilon\beta\gamma \|u^{\epsilon}_{n}\|^{2}.
\end{aligned}
\end{equation*} 
By Young's inequality, the above two inequalities can be estimated as follows:
\begin{equation}\label{2.1-4}
\begin{aligned}
&- \epsilon\alpha \left( u^{\epsilon}_{n} D^{-} u^{\epsilon}_{n}, u^{\epsilon}_{n}\right) + \epsilon\beta \left( u^{\epsilon}_{n}(1-u^{\epsilon}_{n})(u^{\epsilon}_{n}-\gamma), u^{\epsilon}_{n}\right)\\
\leq & -\epsilon\beta \|u^{\epsilon}_{n}\|^{4}_{4}+ \epsilon( 2\alpha + \beta+\beta\gamma)\|u^{\epsilon}_{n}\|^{3}_{3} - \epsilon\beta\gamma \|u^{\epsilon}_{n}\|^{2}\\
\leq & -\epsilon\beta \|u^{\epsilon}_{n}\|^{4}_{4}+ \epsilon(2\alpha+ \beta +\beta\gamma) \sum_{i\in \mathbb{Z}} \left(\frac{\beta}{2\alpha + \beta+\beta\gamma}|u^{\epsilon}_{n,i}|^{4}+ \frac{2\alpha + \beta+\beta\gamma}{4\beta}|u^{\epsilon}_{n,i}|^{2} \right) - \epsilon\beta\gamma \|u^{\epsilon}_{n}\|^{2}\\
\leq & \epsilon \left(\frac{(2\alpha + \beta+\beta\gamma)^2}{4\beta} -\beta\gamma\right) \|u^{\epsilon}_{n}\|^{2}.
\end{aligned}
\end{equation} 
From $\lambda>\lambda^{*}$ in \eqref{2.1-0}, the last term on the right-hand side of the equality \eqref{2.1-1} is estimated as below:
\begin{equation}\label{2.1-5}
\epsilon\left( f, u^{\epsilon}_{n}\right)\leq \frac{\epsilon(\lambda-\lambda^{*})}{2} \|u^{\epsilon}_{n}\|^{2} + \frac{\epsilon}{2(\lambda-\lambda^{*})}\|f\|^{2}.
\end{equation} 
Finally, substituting \eqref{2.1-2}-\eqref{2.1-5} into \eqref{2.1-1}, we get
\begin{equation*}
\|u^{\epsilon}_{n}\|^{2} \leq \frac{1}{2}\|u^{\epsilon}_{n-1}\|^{2} + \frac{1-\epsilon(\lambda-\lambda^{*})}{2}\|u^{\epsilon}_{n}\|^{2} + \frac{\epsilon}{2(\lambda-\lambda^{*})}\|f\|^{2},
\end{equation*} 
which can be reorganized as
\begin{equation}\label{2.1-6}
\|u^{\epsilon}_{n}\|^{2} \leq \frac{1}{1+\epsilon(\lambda-\lambda^{*})}\|u^{\epsilon}_{n-1}\|^{2} + \frac{\epsilon}{\left[1+\epsilon(\lambda-\lambda^{*})\right](\lambda-\lambda^{*})}\|f\|^{2}.
\end{equation}
If $\|u^{\epsilon}_{n-1}\| \leq r^{*}=1+\frac{\|f\|}{\lambda-\lambda^{*}}$, it follows from \eqref{2.1-6} that
\begin{equation}\label{2.1-7}
\begin{aligned}
\|u^{\epsilon}_{n}\|^{2} &\leq \frac{1}{1+\epsilon(\lambda-\lambda^{*})}\left(1+\frac{\|f\|}{\lambda-\lambda^{*}}\right)^{2} + \frac{\epsilon}{\left[1+\epsilon(\lambda-\lambda^{*})\right](\lambda-\lambda^{*})}\|f\|^{2}\\
&\leq 1+2\frac{\|f\|}{\lambda-\lambda^{*}} + \frac{1}{1+\epsilon(\lambda-\lambda^{*})} \left[ \frac{1}{(\lambda-\lambda^{*})^{2}}+\frac{\epsilon}{\lambda-\lambda^{*}}\right]\|f\|^{2}\\
& \leq \left(1+\frac{\|f\|}{\lambda-\lambda^{*}}\right)^{2} = \left(r^{*}\right)^{2},
\end{aligned}
\end{equation}
which indicates $u^{\epsilon}_{n} \in \textit{B}_{r^{*}}$.

Secondly, we prove the existence and uniqueness of the solution. When $n=1$, for any $u_{0}\in \textit{B}_{r^{*}}$ and $\epsilon \in ( 0 ,\epsilon^{*}]$, let
\begin{equation*}
\Phi^{\epsilon}_{u_0} y=u_{0}+\epsilon F y, ~~~\forall y\in \ell^2.
\end{equation*}
If $y\in \textit{B}_{r^{*}+1}$, we deduce from \eqref{2.3} and \eqref{2.1-11} that
\begin{equation*}
\|\Phi^{\epsilon}_{u_0} y\| \leq \|u_{0}\| + \epsilon \|F y\| \leq r^{*} + \epsilon^{*} M_{r^{*}+1} \leq r^{*}+1.
\end{equation*}
Notice that the operator $\Phi^{\epsilon}_{u_0}: \textit{B}_{r^{*}+1}\rightarrow \textit{B}_{r^{*}+1}$ is well defined. From \eqref{2.4} and \eqref{2.1-11}, we get
\begin{equation*}
\|\Phi^{\epsilon}_{u_0} y - \Phi^{\epsilon}_{u_0} z\| = \epsilon \| F y - F z\| \leq \epsilon^{*} L_{r^{*}+1} \|y-z\| \leq  \frac{L_{r^{*}+1}}{1+L_{r^{*}+1}}\|y-z\|,~~~\forall y,z\in \textit{B}_{r^{*}+1}.
\end{equation*}
Since $\frac{L_{r^{*}+1}}{1+L_{r^{*}+1}}<1$, by contraction mapping principle $\Phi^{\epsilon}_{u_0}$ has a unique fixed point $u^{\epsilon}_{1}(u_0)\in  \textit{B}_{r^{*}+1}$. Because $u_0\in  \textit{B}_{r^{*}}$, it can be deduced from \eqref{2.1-7} that $u^{\epsilon}_{1}(u_0)\in  \textit{B}_{r^{*}}$ and it is the unique solution of the IES \eqref{2.2} in the space $\ell^2$.

For all $n\in \mathds{N}$, follow a similar deduction based on the proof process for $n=1$. Assuming the theorem holds for $n=k$, i.e., equation \eqref{2.2} for $n=k$ has a unique solution $u^{\epsilon}_{k}(u_0)\in  \textit{B}_{r^{*}}$. By applying the same method as above, but with $u^{\epsilon}_{k}(u_0)$ instead of $u_0$, we can deduce that equation \eqref{2.2} for $n=k+1$ also has a unique solution $u^{\epsilon}_{k+1}(u_0)\in  \textit{B}_{r^{*}}$. Thus, the proof of the theorem is complete.
\end{proof}

\subsection{Discretization error of the lattice system for the Burgers-Huxley equation}
Using the local Lipschitz continuity and local boundedness of the vector field in Lemma \ref{lemma2.1}, it can be proved that \eqref{2.1} has a local solution 
$u=\left(u_{i}(t)\right)_{i \in \mathbb{Z}}$ for $t\in [0,T_{max})$ in \cite{Klaus-1977}. We will subsequently demonstrate that the local solution extends globally, i.e., $T_{max}=\infty$. 

\begin{lemma}\label{lemma2.3}
For any $u_{0}\in \ell^2$, the lattice system \eqref{2.1} of the Burgers-Huxley equation has a unique solution $u(\cdot,u_{0})\in C\left( [0,\infty),~ \ell^2 \right)$. Furthermore, the ball $\textit{B}_{r^{*}}$ in Theorem \ref{theorem2.1} is positively invariant and absorbing for the lattice system \eqref{2.1}, i.e.
\begin{equation}\label{lemma2.3-1}
u(t,u_{0})\in \textit{B}_{r^{*}},~~~\forall t\geq 0,~~ u_{0}\in \textit{B}_{r^{*}},
\end{equation}
\begin{equation}\label{lemma2.3-2}
\limsup_{t\rightarrow +\infty} \underset{\|u_0\| \leq r}{\sup}\|u(t,u_{0})\|<r^{*},~~~\forall r>0.
\end{equation}
\end{lemma}
\begin{proof}
It is sufficient to provide the bounded estimate of $u(t)$ on the interval $[0,T)$ for each $T>0$. We take the inner product of the lattice system \eqref{2.1} with $u(t)$ to obtain
\begin{equation*}
\frac{1}{2}\frac{d}{dt}\|u(t)\|^{2}=\nu\left(\Lambda u, u\right) - \alpha \left( u D^{-} u, u\right) + \beta \left( u(1-u)(u-\gamma), u\right)-\lambda\left( u, u\right)+\left( f, u\right).
\end{equation*}
The first term on the right-hand side of the above equality is estimated in a manner similar to \eqref{2.1-3}, we get
\begin{equation*}
\nu\left(\Lambda u, u\right) \leq 4\nu \|u\|^{2}.
\end{equation*}
The second and the third terms are estimated by \eqref{2.1-4}, then we have
\begin{equation*}
- \alpha \left( u D^{-} u, u\right) + \beta \left( u(1-u)(u-\gamma), u\right)\leq   \left(\frac{(2\alpha + \beta+\beta\gamma)^2}{4\beta} -\beta\gamma\right) \|u\|^{2}.
\end{equation*}
The last term is estimated by \eqref{2.1-5} as follows:
\begin{equation*}
\left( f, u\right)\leq \frac{(\lambda-\lambda^{*})}{2} \|u\|^{2} + \frac{1}{2(\lambda-\lambda^{*})}\|f\|^{2}.
\end{equation*} 
By using the above inequalities and the symbols in \eqref{2.1-12}, we have
\begin{equation*}
\frac{d}{dt}\|u(t)\|^{2}\leq -(\lambda-\lambda^{*})\|u(t)\|^{2}+\frac{1}{\lambda-\lambda^{*}}\|f\|^{2}. 
\end{equation*} 
By Gronwall lemma, we deduce that
\begin{equation}\label{2.3-3}
\|u(t)\|^{2}\leq e^{-(\lambda-\lambda^{*})t}\|u(0)\|^{2}+\frac{\|f\|^{2}}{(\lambda-\lambda^{*})^{2}}\left(1-e^{-(\lambda-\lambda^{*})t}\right). 
\end{equation} 
According to \eqref{2.3-3}, we get
\begin{equation*}
\max_{0\leq t\leq T}\|u(t)\|^{2}\leq \|u(0)\|^{2}+\frac{\|f\|^{2}}{(\lambda-\lambda^{*})^{2}}<+\infty,~~\forall T>0.
\end{equation*} 
Therefore, the global solution exists. The uniqueness of the global solution is guaranteed by  \eqref{2.4}.

For any $\|u_{0}\|\leq r^{*}=1+\frac{\|f\|}{\lambda-\lambda^{*}}$, it can be deduced from \eqref{2.3-3} that
\begin{equation*}
\begin{aligned}
\|u(t,u_0)\|^{2} & \leq e^{-(\lambda-\lambda^{*})t}\|u_{0}\|^{2}+\frac{\|f\|^{2}}{(\lambda-\lambda^{*})^{2}}\left(1-e^{-(\lambda-\lambda^{*})t}\right) \\
& \leq e^{-(\lambda-\lambda^{*})t}  \left[ 1 + 2\frac{\|f\|}{\lambda-\lambda^{*}} + \frac{\|f\|^{2}}{(\lambda-\lambda^{*})^2} \right] - e^{-(\lambda-\lambda^{*})t}\frac{\|f\|^{2}}{(\lambda-\lambda^{*})^2} + \frac{\|f\|^{2}}{(\lambda-\lambda^{*})^2}  \\
& = e^{-(\lambda-\lambda^{*})t}  +e^{-(\lambda-\lambda^{*})t}2\frac{\|f\|}{\lambda-\lambda^{*}} +  \frac{\|f\|^{2}}{(\lambda-\lambda^{*})^2} \\
& \leq 1 + 2\frac{\|f\|}{\lambda-\lambda^{*}} + \frac{\|f\|^{2}}{(\lambda-\lambda^{*})^2}  = \left(1+\frac{\|f\|}{\lambda-\lambda^{*}} \right)^2 = \left( r^{*} \right)^2.
\end{aligned}
\end{equation*} 
Therefore, $\textit{B}_{r^{*}}$ is positively invariant.

Furthermore, if $\|u_{0}\|\leq r$ for any $r>0$, it follows from \eqref{2.3-3} that
\begin{equation*}
\begin{aligned}
\|u(t,u_0)\|^{2} & \leq e^{-(\lambda-\lambda^{*})t}r^{2}+\frac{\|f\|^{2}}{(\lambda-\lambda^{*})^{2}}\left(1-e^{-(\lambda-\lambda^{*})t}\right) \\
& \leq \frac{1}{2} + \frac{\|f\|^{2}}{(\lambda-\lambda^{*})^2}  < \left( r^{*} \right)^2,~~~\forall t\geq \frac{2\log r +\log \lambda^{*}}{\lambda-\lambda^{*}}.
\end{aligned}
\end{equation*} 
Thus, $\textit{B}_{r^{*}}$ is absorbing. We complete the proof.
\end{proof}

We firstly consider the existence of a global attractor $\mathcal{A}$. This attractor is compact for the continuous-time solution of Burgers-Huxley equation \eqref{2.1}, that is,
\begin{equation*}
u(t,\mathcal{A})=\mathcal{A},~~~\forall t\geq 0 ~ ~and ~ \lim_{t\rightarrow +\infty} d_{\ell^2}\left(u(t,B_{r}),\mathcal{A}\right)=0,~~~\forall r>0.
\end{equation*}

\begin{theorem}\label{theorem3.2}
	The lattice system \eqref{2.1} of the Burgers-Huxley equation has a unique global attractor, which can be expressed as:
	\begin{equation*}
	\mathcal{A}=\bigcap_{T>0}\overline{\bigcup_{t\geq T}u\left(t,\textit{B}_{r^{*}}\right)},
	\end{equation*}
	where $r^{*}=1+\frac{\|f\|}{\lambda-\lambda^{*}}$.
\end{theorem}
\begin{proof}
	According to Lemma \ref{lemma2.3}, the set $\textit{B}_{r^{*}}$ is absorbing for the continuous-time equation. Here, we outline the proof for the asymptotic tails property. We introduce the cut-off function $\xi:\mathbb R^{+}\rightarrow [0,1]$, which is continuously differential and satisfies the following properties:
	\begin{equation*}\xi(s)=
	\left\{\begin{array}{l}
	0,~~0\leq s\leq1,\\
	1,~~s\geq 2.
	\end{array}\right.
	\end{equation*}
	Then there exists a constant $c_{1}>0$ such that for any $k\in \mathbb{N}$ and $i\in \mathbb{Z}$
	\begin{equation}\label{theorem3.1-1}
	|\xi_{k,i+1}-\xi_{k,i}|\leq\frac{c_{1}}{k}, ~~~\text{where}~~\xi_{k,i}=\xi(\frac{|i|}{k}),~~~
	(\xi_{k,i})_{i\in \mathbb{Z}}=:\xi_{k}.
	\end{equation}
	By \eqref{x2.1}, the above inequality can be written as
	\begin{equation}\label{theore3.1-1}
	\|D^{+}\xi_k\|\leq\frac{c_{1}}{k},~~~\forall k\in \mathds{N}.
	\end{equation}
	Let $u=u(t,y)$  be the continuous-time solution of the lattice system \eqref{2.1} for $y\in \textit{B}_{r}$, where $r>0$ is any chosen radius. By taking the inner product of equation \eqref{2.1} with $\xi_{k} u$, we have
	\begin{equation}\label{theorem3.2-1}
	\frac{1}{2}\frac{d}{dt} \sum_{i\in \mathbb{Z}} \xi_{k,i} u_{i}^{2}(t) = \left(\nu\Lambda u-\alpha uD^{-}u+\beta u(1-u)(u-\gamma)-\lambda u + f, \xi_{k}u\right).
	\end{equation}
	By using Young's inequality, the first term can be infer that
	\begin{equation}\label{theorem3.2-3}
	\begin{aligned}
	\nu\left(\Lambda u, \xi_{k}u\right) & = \nu \sum_{i\in \mathbb{Z}} \xi_{k,i}u_{i}\left( - u_{i+1} + 2 u_{i}- u_{i-1}\right)\\
	& = \nu \sum_{i\in \mathbb{Z}} \xi_{k,i}\left( - u_{i}u_{i+1} + 2u_{i}^{2}- u_{i-1}u_{i}\right)\\
	& \leq \nu \sum_{i\in \mathbb{Z}} \xi_{k,i} \left( \frac{1}{2}u_{i}^{2} + \frac{1}{2}u_{i+1}^{2}+ 2 u_{i}^{2} + \frac{1}{2}u_{i-1}^{2} + \frac{1}{2}u_{i}^{2}\right)\\
	& = 3\nu \sum_{i\in \mathbb{Z}} \xi_{k,i}u_{i}^{2} + \frac{1}{2}\nu \sum_{i\in \mathbb{Z}} \xi_{k,i}u_{i-1}^{2} +  \frac{1}{2}\nu \sum_{i\in \mathbb{Z}} \xi_{k,i}u_{i+1}^{2},
	\end{aligned}
	\end{equation} 
	where
	\begin{equation}
	\begin{aligned}
	\frac{1}{2}\nu \sum_{i\in \mathbb{Z}} \xi_{k,i}u_{i-1}^{2} & = \frac{1}{2}\nu \left[\sum_{i\in \mathbb{Z}} \left(\xi_{k,i}-\xi_{k,i-1}\right)u_{i-1}^{2} + \sum_{i\in \mathbb{Z}} \xi_{k,i-1}u_{i-1}^{2}\right]\\
	& = \frac{1}{2}\nu \left[\sum_{i\in \mathbb{Z}} \left(D^{+}\xi_{k}\right)_{i-1}u_{i-1}^{2} + \sum_{i\in \mathbb{Z}} \xi_{k,i-1}u_{i-1}^{2}\right]\\
	& = \frac{1}{2}\nu \left[\sum_{i\in \mathbb{Z}} \left(D^{+}\xi_{k}\right)_{i}u_{i}^{2} + \sum_{i\in \mathbb{Z}} \xi_{k,i}u_{i}^{2}\right]\\
	& \leq \frac{1}{2}\nu \frac{c_{1}}{k}\left(r^{*}\right)^{2} + \frac{1}{2}\nu\sum_{i\in \mathbb{Z}} \xi_{k,i}u_{i}^{2},
	\end{aligned}
	\end{equation} 
	and
	\begin{equation}\label{theorem3.2-4}
	\begin{aligned}
	\frac{1}{2}\nu \sum_{i\in \mathbb{Z}} \xi_{k,i}u_{i+1}^{2} & = \frac{1}{2}\nu \left[\sum_{i\in \mathbb{Z}} \left(\xi_{k,i}-\xi_{k,i+1}\right)u_{i+1}^{2} + \sum_{i\in \mathbb{Z}} \xi_{k,i+1}u_{i+1}^{2}\right]\\
	& = \frac{1}{2}\nu \left[-\sum_{i\in \mathbb{Z}} \left(D^{+}\xi_{k}\right)_{i}u_{i+1}^{2} + \sum_{i\in \mathbb{Z}} \xi_{k,i}u_{i}^{2}\right]\\
	& \leq \frac{1}{2}\nu \frac{c_{1}}{k}\left(r^{*}\right)^{2} + \frac{1}{2}\nu\sum_{i\in \mathbb{Z}} \xi_{k,i}u_{i}^{2}.
	\end{aligned}
	\end{equation} 
	As above, we can derive the first inner product on the right-hand side as follows:
	\begin{equation*}
	\begin{aligned}
	\left(\nu\Lambda u, \xi_{k}u\right) & \leq 3\nu \sum_{i\in \mathbb{Z}} \xi_{k,i}u_{i}^{2}(t)+\frac{1}{2}\nu \frac{c_{1}}{k}\|u\|^{2} + \frac{1}{2}\nu\sum_{i\in \mathbb{Z}} \xi_{k,i}u_{i}^{2}(t)+\frac{1}{2}\nu \frac{c_{1}}{k}\|u\|^{2} + \frac{1}{2}\nu\sum_{i\in \mathbb{Z}} \xi_{k,i}u_{i}^{2}(t)\\
	& = 4\nu \sum_{i\in \mathbb{Z}} \xi_{k,i}u_{i}^{2}(t) + \nu\frac{c_{1}}{k}\|u\|^{2}.
	\end{aligned}
	\end{equation*}
	The second inner product is estimated as below:
	\begin{equation*}
	\begin{aligned}
	-\alpha\left( uD^{-}u, \xi_{k}u\right) &= \alpha\sum_{i\in \mathbb{Z}} \xi_{k,i} u_{i}^{3}(t) - \alpha\sum_{i\in \mathbb{Z}} \xi_{k,i} u_{i}^{2}(t)u_{i-1}(t)\\ 
	&\leq \alpha\sum_{i\in \mathbb{Z}} \xi_{k,i} |u_{i}(t)|^{3} + \alpha\sum_{i\in \mathbb{Z}} \xi_{k,i}\left[ \frac{2}{3}|u_{i}(t)|^{3}+\frac{1}{3}|u_{i-1}(t)|^{3}\right]\\
	&= \frac{5}{3}\alpha\sum_{i\in \mathbb{Z}} \xi_{k,i} |u_{i}(t)|^{3} + \frac{1}{3}\alpha\sum_{i\in \mathbb{Z}}\left[ (\xi_{k,i}-\xi_{k,i-1}) |u_{i-1}(t)|^{3}+\xi_{k,i-1}|u_{i-1}(t)|^{3}\right]\\
	&= 2\alpha\sum_{i\in \mathbb{Z}} \xi_{k,i} |u_{i}(t)|^{3} + \frac{1}{3}\alpha\sum_{i\in \mathbb{Z}}\left(D^{+}\xi_{k}\right)_{i-1} |u_{i-1}(t)|^{3}\\
	&\leq  2\alpha\sum_{i\in \mathbb{Z}} \xi_{k,i} |u_{i}(t)|^{3} + \frac{c_1\alpha}{3k}\|u(t)\|_3^{3}.
	\end{aligned}
	\end{equation*}
	Writing the third term in component form, we have
	\begin{equation*}
	\begin{aligned}
	\beta \left(u(1-u)(u-\gamma),\xi_{k}u \right) &= -\beta \sum_{i\in \mathbb{Z}} \xi_{k,i} u_{i}^{4}(t) + \beta (1+\gamma)\sum_{i\in \mathbb{Z}} \xi_{k,i} u_{i}^{3}(t) -\beta \gamma \sum_{i\in \mathbb{Z}} \xi_{k,i} u_{i}^{2}(t).
	\end{aligned}
	\end{equation*}
	Combining the above two items, we have
	\begin{equation*}
	\begin{aligned}
	&-\alpha\left( uD^{-}u + \beta u(1-u)(u-\gamma), \xi_{k}u\right) \\
	\leq & (2\alpha + \beta + \beta\gamma)\sum_{i\in \mathbb{Z}} \xi_{k,i} |u_{i}(t)|^{3} + \frac{c_1\alpha}{3k}\|u(t)\|_3^{3} -\beta \sum_{i\in \mathbb{Z}} \xi_{k,i} u_{i}^{4}(t) -\beta \gamma \sum_{i\in \mathbb{Z}} \xi_{k,i} u_{i}^{2}(t)\\
	\leq & \beta \sum_{i\in \mathbb{Z}} \xi_{k,i} u_{i}^{4}(t) + \frac{(2\alpha + \beta + \beta\gamma)^2}{4\beta} \sum_{i\in \mathbb{Z}} \xi_{k,i} u_{i}^{2}(t) + \frac{c_1\alpha}{3k}\|u(t)\|_3^{3} -\beta \sum_{i\in \mathbb{Z}} \xi_{k,i} u_{i}^{4}(t)-\beta \gamma \sum_{i\in \mathbb{Z}} \xi_{k,i} u_{i}^{2}(t) \\
	= & \frac{(2\alpha + \beta + \beta\gamma)^2}{4\beta} \sum_{i\in \mathbb{Z}} \xi_{k,i} u_{i}^{2}(t) + \frac{c_1\alpha}{3k}\|u(t)\|_3^{3} -\beta \gamma \sum_{i\in \mathbb{Z}} \xi_{k,i} u_{i}^{2}(t).
	\end{aligned}
	\end{equation*}
	The last inner product is estimated as follows:
	\begin{equation*}
	\begin{aligned}
	\left(f, \xi_{k}u\right) & = \sum_{i\in \mathbb{Z}} \xi_{k,i} u_{i}(t)f_{i}
	\leq  \frac{\left(\lambda-\lambda^{*}\right)}{2}\sum_{i\in \mathbb{Z}} \xi_{k,i} u_{i}^{2}(t)+\frac{1}{2\left(\lambda-\lambda^{*}\right)}\sum_{i\in \mathbb{Z}} \xi_{k,i} f_{i}^{2}.
	\end{aligned}
	\end{equation*} 
	Substituting the above estimates into formula \eqref{theorem3.2-1}, we have
	\begin{equation*}
	\frac{d}{dt} \sum_{i\in \mathbb{Z}} \xi_{k,i} u_{i}^{2}(t) + (\lambda-\lambda^{*})\sum_{i\in \mathbb{Z}} \xi_{k,i} u_{i}^{2}(t) \leq \frac{2c_{1}\nu}{k}\|u\|^{2}+ \frac{2c_{1}\alpha}{3k}\|u\|_3^{3}+ \frac{1}{\lambda-\lambda^{*}}\sum_{|i|\geq k}f^{2}_{i}.
	\end{equation*}  
	By Gronwall's inequality, for any $u(0)\in \textit{B}_{r}$ we conclude that
	\begin{equation}\label{theorem3.2-2}
	\sum_{i\in \mathbb{Z}} \xi_{k,i} u_{i}^{2}(t)\leq e^{-(\lambda-\lambda^{*})t}r^{2} + \frac{1}{\lambda-\lambda^{*}}\left[\frac{2c_{1}\nu}{k}\|u\|^{2}+ \frac{2c_{1}\alpha}{3k}\|u\|_3^{3}+ \frac{1}{\lambda-\lambda^{*}}\sum_{|i|\geq k}f^{2}_{i}\right].
	\end{equation}  
	Because $\textit{B}_{r^{*}}$ is attracting, $\|u\|^{2}$ and $\|u\|_3^{3}$ are bounded as $t\rightarrow +\infty$. Therefore, from \eqref{theorem3.2-2}, we deduce that
	\begin{equation*}
	\lim_{k,t\rightarrow +\infty}\sup_{y\in \textit{B}_{r}}\sum_{|i|\geq 2k} u_{i}^{2}(t,y) \leq \lim_{k,t\rightarrow +\infty}\sup_{y\in \textit{B}_{r}} \sum_{i\in \mathbb{Z}} \xi_{k,i} u_{i}^{2}(t) = 0.
	\end{equation*}  
	The asymptotic tails property of the continuous-time solution is proven.  
\end{proof}

In order to consider discretization error of the lattice system, we introduce a one-step second-order Taylor expansion for the solution of equation \eqref{2.1}.
\begin{lemma}\label{lemma2.4}
For $n\in N$, and for any times $t_{n},~t_{n-1}\geq 0$ satisfied $t_{n}-t_{n-1}=\epsilon>0$, there exists an operator $\Theta_{n,\epsilon}:\textit{B}_{r^{*}}\rightarrow \ell^2$ such that
\begin{equation}\label{2.4-1}
u(t_{n-1},y)=u(t_{n},y)-\epsilon Fu(t_{n},y)+\epsilon^{2}\Theta_{n,\epsilon}y,~~~\forall y\in \textit{B}_{r^{*}}, 
\end{equation}
where $F$ represents the vector field in \eqref{2.1}. Furthermore, the operator $\Theta_{n,\epsilon}$ is uniformly bounded on $\textit{B}_{r^{*}}$, i.e.,
\begin{equation}\label{2.4-2}
\|\Theta_{n,\epsilon}y\|\leq \frac{1}{2}L_{r^{*}}M_{r^{*}}, 
\end{equation}
where the constants $L_{r^{*}}$ and $M_{r^{*}}$  are defined in Lemma \ref{lemma2.1} and independent of $n,\epsilon,u_{0}$.
\end{lemma}

In the following theorem, we give the one-step error between the continuous-time solution and its discrete-time solution, demonstrating that the discrete-time solution converges to the continuous-time solution as $\epsilon\rightarrow 0^{+}$.
\begin{theorem}\label{theorem2.5}
Denoting by $u(t,y)$ and $u^{\epsilon}_{n}(y)$ the solutions to \eqref{2.1} and \eqref{2.2} with the initial condition $y\in \textit{B}_{r^{*}}$, we get
 \begin{equation}\label{2.5-1}
\|u(\epsilon,u^{\epsilon}_{n-1}(y))-u^{\epsilon}_{n}(y)\|\leq L_{r^{*}}M_{r^{*}} L_{r^{*}+1}\epsilon^2, ~~~\forall\epsilon\in(0,\epsilon^{*}],~~n\in\mathds N.
\end{equation}
Furthermore, for every $T > 0$, one has
\begin{equation}\label{2.5-2}
\|u(t_{n},y)-u^{\epsilon}_{n}(y)\|\leq \frac{M_{r^{*}}}{2}e^{L_{r^{*}}T} \epsilon, ~~~\forall t_{n}=\epsilon n\in(0,T].
\end{equation}
\end{theorem}

The proofs of Lemma \ref{lemma2.4} and Theorem \ref{theorem2.5} can be found in \cite{Liu-2024}, where $F$ represents the vector field of the Burgers-Huxley equation.

\section{Numerical attractor for the IES of Burgers-Huxley equation}\label{sec:3}

\subsection{Existence of a unique numerical attractor for the IES}
The discrete-time solution $u^{\epsilon}_{n}$ in Theorem \ref{theorem2.1} defines a discrete dynamical system on the ball $\textit{B}_{r^{*}}$, expressed as:
\begin{equation*}
S_{\epsilon}(n):\textit{B}_{r^{*}}\rightarrow\textit{B}_{r^{*}},~~S_{\epsilon}(n)y=u^{\epsilon}_{n}(y),~~~y\in \textit{B}_{r^{*}},~~n \in \mathds{N},
\end{equation*}
for any $\epsilon\in(0,\epsilon^{*}]$. Moreover, it is easy to obtain $S_{\epsilon}(n+m)=S_{\epsilon}(n)S_{\epsilon}(m)$.

A compact set $\mathcal{A}^{\epsilon}$ in $\ell^2$ is called a numerical attractor of the IES \eqref{2.2} for the Burgers-Huxley equation if $\mathcal{A}^{\epsilon}$ satisfies the following two conditions:\\
1. The compact set $\mathcal{A}^{\epsilon}$ is invariant under the mapping $S_{\epsilon}(n)$, which means that
\begin{equation*}
S_{\epsilon}(n)\mathcal{A}^{\epsilon}=\mathcal{A}^{\epsilon},~~~n \in \mathds{N}.
\end{equation*}
2. The compact set $\mathcal{A}^{\epsilon}$ is attracting, which means that
\begin{equation*}
\lim_{n\rightarrow\infty}d_{\ell^2}\left(S_{\epsilon}(n)\textit{B}_{r^{*}}, \mathcal{A}^{\epsilon}\right)=0,
\end{equation*}  
where $d_{\ell^2}$ denotes the Hausdorff semi-distance, defined as $d(A,B)=\sup\limits_{a\in A}\inf\limits_{b\in B}\|a-b\|$. 

\begin{theorem}\label{theorem3.1}
For any $\epsilon\in (0,\epsilon^{*}]$, the IES \eqref{2.2} for the Burgers-Huxley equation has a unique numerical attractor, which can be expressed as:
\begin{equation*}
\mathcal{A}^{\epsilon}=\bigcap_{m\in \mathds{N}}\overline{\bigcup_{n\geq m}S_{\epsilon}(n)\textit{B}_{r^{*}}}.
\end{equation*}
\end{theorem}
\begin{proof}
Because $\textit{B}_{r^{*}}$ is absorbing for the discrete dynamical system, and following a similar result as stated in \cite{Li-2023-2, Han-2020}, it is sufficient to prove the asymptotic tails property. That is for any $\eta>0$, there exist $N,I\in \mathds{N}$ such that the following inequality holds:
\begin{equation}\label{theore3.1-0}
\sum_{|i|\geq I}|u^{\epsilon}_{n,i}(y)|^{2}<\eta,~~~\forall n\geq N,~~y\in \textit{B}_{r^{*}},
\end{equation}
where $u^{\epsilon}_{n}(y)=\left( u^{\epsilon}_{n,i}(y)\right)_{i\in \mathds{Z}}$ is the unique solution of \eqref{2.2}. For convenience in the proof, we will omit $\epsilon,(y)$, and use the concise notation $u_n$ instead of $u^{\epsilon}_{n}(y)$ in the following.

Here, we outline the proof for the asymptotic tails property, utilizing the notations introduced in Theorem \ref{theorem3.2}.

We take the inner product of \eqref{2.2} with $\xi_{k} u^{\epsilon}_{n}$ in $\ell^2$ , then we get 
\begin{equation}\label{3.1-3}
\begin{aligned}
\sum_{i\in \mathbb{Z}}\xi_{k,i}u^{2}_{n,i}=\left( u_{n-1},\xi_{k}u_{n}\right)+\epsilon\left(\nu\Lambda u_{n} - \alpha u_{n} D^{-} u_{n} + \beta u_{n}(1-u_{n})(u_{n}-\gamma) -\lambda u_{n}+ f, \xi_{k}u_{n}\right).
\end{aligned}
\end{equation}
	The first term on the right-hand side of the equality \eqref{3.1-3} is estimated as below:
	\begin{equation}\label{3.1-4}
	\left( u_{n-1}, \xi_{k}u_{n}\right) =\sum_{i\in \mathbb{Z}}\xi_{k,i}u_{n,i}u_{n-1,i}\leq \frac{1}{2}\sum_{i\in \mathbb{Z}}\xi_{k,i}u_{n-1,i}^2 + \frac{1}{2}\sum_{i\in \mathbb{Z}}\xi_{k,i}u_{n,i}^2. 
	\end{equation} 
Using the same method as in the derivation of \eqref{theorem3.2-3}-\eqref{theorem3.2-4}, there holds
\begin{equation}\label{3.1-5}
\begin{aligned}
\epsilon\nu\left(\Lambda u_{n}, \xi_{k}u_{n}\right) & \leq 3\epsilon\nu \sum_{i\in \mathbb{Z}} \xi_{k,i}u_{n,i}^{2}+\frac{1}{2}\epsilon\nu \frac{c_{1}}{k}\left(r^{*}\right)^{2} + \frac{1}{2}\epsilon\nu\sum_{i\in \mathbb{Z}} \xi_{k,i}u_{n,i}^{2}+\frac{1}{2}\epsilon\nu \frac{c_{1}}{k}\left(r^{*}\right)^{2} + \frac{1}{2}\epsilon\nu\sum_{i\in \mathbb{Z}} \xi_{k,i}u_{n,i}^{2}\\
& = 4\epsilon\nu \sum_{i\in \mathbb{Z}} \xi_{k,i}u_{n,i}^{2} + \epsilon\nu\frac{c_{1}}{k}\left(r^{*}\right)^{2}.
\end{aligned}
\end{equation} 
The third term is estimated using a similar method as follows:
\begin{equation}\label{3.1-55}
\begin{aligned}
- \epsilon\alpha \left( u_{n} D^{-} u_{n}, \xi_{k}u_{n}\right) & = - \epsilon\alpha  \sum_{i\in \mathbb{Z}} \xi_{k,i}\left( u_{n,i}^{2}u_{n,i-1} - u_{n,i}^{3}\right)\\
& \leq \epsilon\alpha  \sum_{i\in \mathbb{Z}} \xi_{k,i}|u_{n,i}|^{3} + \epsilon\alpha \sum_{i\in \mathbb{Z}} \xi_{k,i}\left( \frac{2}{3}|u_{n,i}|^3 + \frac{1}{3} |u_{n,i-1}|^3\right)\\
& =\frac{5}{3}\epsilon\alpha  \sum_{i\in \mathbb{Z}} \xi_{k,i}|u_{n,i}|^{3} + \frac{1}{3}\epsilon\alpha \sum_{i\in \mathbb{Z}} \xi_{k,i}|u_{n,i-1}|^3\\
& =\frac{5}{3}\epsilon\alpha  \sum_{i\in \mathbb{Z}} \xi_{k,i}|u_{n,i}|^{3} + \frac{1}{3}\epsilon\alpha \sum_{i\in \mathbb{Z}} \left[ \xi_{k,i-1}|u_{n,i-1}|^3 + \left(\xi_{k,i}-\xi_{k,i-1}\right)|u_{n,i-1}|^3\right]\\
& \leq 2\epsilon\alpha  \sum_{i\in \mathbb{Z}} \xi_{k,i}|u_{n,i}|^{3} + \frac{c_{1}\epsilon\alpha }{3k}\|u\|_3^{3} \\
& \leq 2\epsilon\alpha \sum_{i\in \mathbb{Z}} \xi_{k,i}|u_{n,i}|^{3} + \frac{c_{1}\epsilon\alpha}{3k}\left(r^{*}\right)^{3}.
\end{aligned}
\end{equation} 
The fourth term is reorganized into the following form:
\begin{equation}\label{3.1-6}
\begin{aligned}
\epsilon\beta \left( u_{n}(1-u_{n})(u_{n}-\gamma),\xi_{k} u_{n}\right) & = \epsilon\beta \left( -u_{n}^{3}+(1+\gamma) u_{n}^{2} - \gamma u_{n}, \xi_{k}u_{n} \right)\\
& \leq -\epsilon\beta \sum_{i\in \mathbb{Z}} \xi_{k,i} u_{n,i}^{4}+\epsilon\beta(1+\gamma)\sum_{i\in \mathbb{Z}} \xi_{k,i} |u_{n,i}|^{3} - \epsilon\beta\gamma \sum_{i\in \mathbb{Z}} \xi_{k,i} u_{n,i}^{2}.
\end{aligned}
\end{equation} 
From inequality \eqref{2.1-0}, the last term on the right-hand side of the equality \eqref{3.1-3} is estimated as below:
\begin{equation}\label{3.1-7}
\begin{aligned}
\epsilon\left(f, \xi_{k}u_{n}\right) & = \epsilon \sum_{i\in \mathbb{Z}} \xi_{k,i} u_{n,i}f_{i}
\leq  \frac{\epsilon\left(\lambda-\lambda^{*}\right)}{2}\sum_{i\in \mathbb{Z}} \xi_{k,i} u_{n,i}^{2}+\frac{\epsilon}{2\left(\lambda-\lambda^{*}\right)}\sum_{i\in \mathbb{Z}} \xi_{k,i} f_{i}^{2}.
\end{aligned}
\end{equation} 
Finally, by substituting \eqref{3.1-4}-\eqref{3.1-7} into \eqref{3.1-3}, we have
\begin{equation}\label{3.1-8}
\begin{aligned}
\sum_{i\in \mathbb{Z}} \xi_{k,i} u_{n,i}^{2} &\leq \frac{1}{2}\sum_{i\in \mathbb{Z}} \xi_{k,i} u_{n-1,i}^{2} + \left[\frac{1}{2} + 4\epsilon\nu - \epsilon\beta\gamma - \epsilon \lambda + \frac{\epsilon\left(\lambda-\lambda^{*}\right)}{2}\right]\sum_{i\in \mathbb{Z}} \xi_{k,i} u_{n,i}^{2}\\
&~~~~+\epsilon\left( 2\alpha + \beta + \beta\gamma \right) \sum_{i\in \mathbb{Z}} \xi_{k,i} |u_{n,i}|^{3} - \epsilon\beta \sum_{i\in \mathbb{Z}} \xi_{k,i} u_{n,i}^{4}
+\frac{\epsilon}{2\left(\lambda-\lambda^{*}\right)}\sum_{i\in \mathbb{Z}} \xi_{k,i} f_{i}^{2} + \epsilon\nu\frac{c_{1}}{k}\left(r^{*}\right)^{2} + \frac{c_{1}\epsilon\alpha}{3k}\left(r^{*}\right)^{3}\\
&\leq  \frac{1}{2}\sum_{i\in \mathbb{Z}} \xi_{k,i} u_{n-1,i}^{2}+\left[\frac{1}{2} + 4\epsilon\nu - \epsilon\beta\gamma - \epsilon \lambda + \frac{\epsilon\left(\lambda-\lambda^{*}\right)}{2}\right]\sum_{i\in \mathbb{Z}} \xi_{k,i} u_{n,i}^{2}\\
&~~~+\epsilon\beta\sum_{i\in \mathbb{Z}} \xi_{k,i} u_{n,i}^{4}+\frac{\epsilon\left( 2\alpha + \beta + \beta\gamma \right)^2}{4\beta} \sum_{i\in \mathbb{Z}} \xi_{k,i} u_{n,i}^{2} - \epsilon\beta \sum_{i\in \mathbb{Z}} \xi_{k,i} u_{n,i}^{4}\\
&~~~+\frac{\epsilon}{2\left(\lambda-\lambda^{*}\right)}\sum_{i\in \mathbb{Z}} \xi_{k,i} f_{i}^{2} + \epsilon\nu\frac{c_{1}}{k}\left(r^{*}\right)^{2} + \frac{c_{1}\epsilon\alpha}{3k}\left(r^{*}\right)^{3}\\
&\leq \frac{1}{2}\sum_{i\in \mathbb{Z}} \xi_{k,i} u_{n-1,i}^{2} + \frac{1-\epsilon(\lambda-\lambda^{*})}{2}\sum_{i\in \mathbb{Z}} \xi_{k,i} u_{n,i}^{2}+\epsilon\left[\frac{c_{2}}{k}+\frac{1}{2(\lambda-\lambda^{*})}\sum_{|i|\geq k}f_{i}^{2}\right],
\end{aligned}
\end{equation} 
where $c_{2}=c_{1}\nu\left(r^{*}\right)^2+\frac{c_{1}\alpha}{3}\left(r^{*}\right)^3$. As the last term in \eqref{3.1-8} tends to zero, we can select $k_{\eta}>0$ such that
\begin{equation*}
\frac{c_{2}}{k}+\frac{1}{2(\lambda-\lambda^{*})}\sum_{|i|\geq k}f_{i}^{2}<\frac{\eta(\lambda-\lambda^{*})}{4},~~~\forall k\geq k_{\eta}.
\end{equation*}
Thus, by a straightforward calculation, inequality \eqref{3.1-8} takes the following form
\begin{equation}\label{3.1-9}
\sum_{i\in \mathbb{Z}} \xi_{k,i} u_{n,i}^{2} \leq \frac{1}{1+\epsilon(\lambda-\lambda^{*})}\sum_{i\in \mathbb{Z}} \xi_{k,i} u_{n-1,i}^{2} + \frac{\epsilon(\lambda-\lambda^{*})}{1+\epsilon(\lambda-\lambda^{*})} \frac{\eta}{2}.
\end{equation}
Based on the recurrence relation, as $n\rightarrow \infty$ for any $k\geq k_{\eta}$, we get
\begin{equation*}
\begin{aligned}
\sum_{i\in \mathbb{Z}} \xi_{k,i} u_{n,i}^{2} & \leq \frac{1}{\left[1+\epsilon(\lambda-\lambda^{*})\right]^{n}}\sum_{i\in \mathbb{Z}} \xi_{k,i} u_{0,i}^{2} + \frac{\eta}{2}\sum_{m=1}^{n}\frac{\epsilon(\lambda-\lambda^{*})}{\left[1+\epsilon(\lambda-\lambda^{*})\right]^m}\\
&\leq \frac{(r^{*})^2}{\left[1+\epsilon(\lambda-\lambda^{*})\right]^{n}} + \frac{\eta}{2} \rightarrow \frac{\eta}{2}.
\end{aligned}
\end{equation*}
Thus, we have
\begin{equation*}
\sum_{|i|\geq 2k} u_{n,i}^{2} \leq \sum_{i\in \mathbb{Z}} \xi_{k,i} u_{n,i}^{2}<\eta,~~~\forall n\geq N,~~k\geq k_{\eta}.
\end{equation*}
Chooseing $I=2k_{\eta}$, and the theorem is proved.
\end{proof}

\subsection{Approximation from numerical attractor to the global attractor}

Below, we present a theorem concerning the convergence of the numerical attractor to the global attractor under the Hausdorff semi-distance.
\begin{theorem}\label{theorem3.3}
Let $\mathcal{A}$ denote the global attractor of the lattice system \eqref{2.1}, and let $\mathcal{A}^{\epsilon}$ denote the numerical attractor of the IES \eqref{2.2} for the Burgers-Huxley equation. Then, the following conclusions hold:\\
(i) The numerical attractor $\mathcal{A}^{\epsilon}$ is upper semi-convergence to the global attractor $\mathcal{A}$ under the Hausdorff semi-distance, that is,
\begin{equation*}
\lim_{\epsilon\rightarrow 0^{+}}d_{\ell^2}\left(\mathcal{A}^\epsilon,\mathcal{A} \right)=0.
\end{equation*}
(ii) The numerical attractor $\mathcal{A}^{\epsilon}$ is upper semi-continuous on $(0,\epsilon^{*}]$, that is
\begin{equation*}
\lim_{\epsilon\rightarrow \epsilon_0}d_{\ell^2}\left(\mathcal{A}^\epsilon,\mathcal{A}^{\epsilon_0} \right)=0,~~~\forall \epsilon_0\in (0,\epsilon^{*}].
\end{equation*}
\end{theorem}
\begin{proof}
	Ever since \cite{Han-2020} the proof of this type of upper semicontinuous have been attached great importance to and improved all along. Now the method to prove it is more or less standard . Broadly speaking, supplied with discretization error in Theorem \ref{theorem2.5},  
	the unique existence of solutions (Lemma \ref{lemma2.3}) and unique  global attractor (Theorem \ref{theorem3.2}) for Burgers-Huxley equation, as well as the unique existence of solutions (Theorem \ref{theorem2.1}) and unique numerical attractor (Theorem \ref{theorem3.1}) for discrete-time Burgers-Huxley equation.
For a detailed proof, please refer to \cite{Han-2020,Li-2023-2,Liu-2024}.
\end{proof}

\section{Finite-dimensional approximation of numerical attractor}\label{sec:5}
\subsection{Existence of truncated numerical attractor}
We consider the finite-dimensional approximation for the IES \eqref{2.2}. For $m\in\mathds{N}$ we study the system of implicit difference equations in $\mathbb R^{2m + 1}$ by truncating the infinite-dimensional system as follows: 
\begin{equation}\label{4.1}
\left\{ \begin{array}{cl}
u_{n,-m}^{\epsilon,m}=&u_{n-1,-m}^{\epsilon,m}+\epsilon\left[ \nu\left( 2u_{n,-m}^{\epsilon,m}-u_{n,-m+1}^{\epsilon,m}\right)+\alpha \left(u_{n,-m}^{\epsilon,m}
\right)^2 \right.\vspace{0.1cm}\\
&\left. +\beta u_{n,-m}^{\epsilon,m}\left(1-u_{n,-m}^{\epsilon,m} \right)\left(u_{n,-m}^{\epsilon,m} -\gamma \right)-\lambda u_{n,-m}^{\epsilon,m}+f_{-m}\right],\\
u_{n,-m+1}^{\epsilon,m}=&u_{n-1,-m+1}^{\epsilon,m}+\epsilon\left[ \nu\left( -u_{n,-m}^{\epsilon,m}+2u_{n,-m+1}^{\epsilon,m}-u_{n,-m+2}^{\epsilon,m}\right)-\alpha u_{n,-m+1}^{\epsilon,m} \left(u_{n,-m}^{\epsilon,m}-u_{n,-m+1}^{\epsilon,m}\right) \right.\vspace{0.1cm}\\
&\left. +\beta u_{n,-m+1}^{\epsilon,m}\left(1-u_{n,-m+1}^{\epsilon,m} \right)\left(u_{n,-m+1}^{\epsilon,m} -\gamma \right)-\lambda u_{n,-m+1}^{\epsilon,m}+f_{-m+1}\right],\\
& {\mspace{6mu}\mspace{6mu}\mspace{6mu}\mspace{6mu}\mspace{6mu}\mspace{6mu}\mspace{6mu}\mspace{6mu}\mspace{6mu}\mspace{6mu}\mspace{6mu}\mspace{6mu}\mspace{6mu}\mspace{6mu}\mspace{6mu}\mspace{6mu} \vdots} \\
u_{n,m-1}^{\epsilon,m}=&u_{n-1,m-1}^{\epsilon,m}+\epsilon\left[ \nu\left( -u_{n,m-2}^{\epsilon,m}+2u_{n,m-1}^{\epsilon,m}-u_{n,m}^{\epsilon,m}\right)-\alpha u_{n,m-1}^{\epsilon,m} \left(u_{n,m-2}^{\epsilon,m}-u_{n,m-1}^{\epsilon,m}\right) \right.\vspace{0.1cm}\\
&\left. +\beta u_{n,m-1}^{\epsilon,m}\left(1-u_{n,m-1}^{\epsilon,m} \right)\left(u_{n,m-1}^{\epsilon,m} -\gamma \right)-\lambda u_{n,m-1}^{\epsilon,m}+f_{m-1}\right], \\
u_{n,m}^{\epsilon,m}=&u_{n-1,m}^{\epsilon,m}+\epsilon\left[ \nu\left( -u_{n,m-1}^{\epsilon,m}+2u_{n,m}^{\epsilon,m}\right)-\alpha u_{n,m}^{\epsilon,m} \left(u_{n,m-1}^{\epsilon,m}-u_{n,m}^{\epsilon,m}\right) \right.\vspace{0.1cm}\\
&\left. +\beta u_{n,m}^{\epsilon,m}\left(1-u_{n,m}^{\epsilon,m} \right)\left(u_{n,m}^{\epsilon,m} -\gamma \right)-\lambda u_{n,m}^{\epsilon,m}+f_{m}\right].
\end{array} \right.
\end{equation}
The initial condition is given by
\begin{equation*}
u_{0}^{\epsilon,m} = z \in \mathbb{R}^{2m + 1}.
\end{equation*}
We apply the Dirichlet boundary condition as in \cite{Liu-2024}, where $u_{n,-m-1}^{\epsilon,m}=u_{n,m+1}^{\epsilon,m}=0$ are used in the first and last equations of \eqref{4.1}.

Suppose
\begin{equation} \label{4.1-1}
D^{-}_m= \begin{pmatrix}
{- 1} & 0 & 0 & \cdots & 0 \\
1 & {- 1} & 0 & \cdots & 0 \\
0 & 1 & {- 1} & \cdots & 0 \\
\vdots & \ddots  & \ddots &\ddots  & \vdots   \\
0 & 0 & \cdots & 1 & {- 1}
\end{pmatrix}\in  \left(\mathbb{R}^{(2m + 1)}\right)^2,
\end{equation}
then we have
\begin{equation} \label{4.1-2}
\Lambda_{m} = D^{+}_m D^{-}_m= \begin{pmatrix}
2 & {- 1} & 0 & \cdots & 0 & 0 & 0 \\
{- 1} & 2 & {- 1} & \cdots & 0 & 0 & 0 \\
0 & {- 1} & 2 & \cdots & 0 & 0 & 0 \\
\vdots & \vdots  & \vdots & \ddots  & \vdots &\vdots  & \vdots   \\
0 & 0 & 0 & \cdots & {- 1} & 2 & {- 1} \\
0 & 0 & 0 & \cdots & 0 & {- 1} & 1
\end{pmatrix}.
\end{equation}
Let $f^{m}=(f_{i})_{|i|\leq m} \in \mathbb{R}^{(2m + 1)}$.
The truncated system \eqref{4.1} can then be written as follows:
\begin{equation}\label{4.2}
\left\{\begin{array}{cl}
&u_{n}^{\epsilon,m}= u_{n-1}^{\epsilon,m} + \epsilon \left( \nu \Lambda_{m} u_{n}^{\epsilon,m}
-\alpha u_{n}^{\epsilon,m} D^{-}_{m} u_{n}^{\epsilon,m} + \beta u_{n}^{\epsilon,m}(1-u_{n}^{\epsilon,m})(u_{n}^{\epsilon,m}-\gamma) -\lambda u_{n}^{\epsilon,m}+f^{m}\right),\\
&u_{0}^{\epsilon,m} = z \in \mathbb{R}^{2m + 1}.
\end{array} \right.
\end{equation}
The truncated vector field is also expressed as:
\begin{equation*}
F_{m}z = \nu \Lambda_{m} z
-\alpha z D^{-}_{m} z + \beta z(1-z)(z-\gamma) -\lambda z + f^{m},~~~\forall z \in \mathbb{R}^{2m + 1}.
\end{equation*}

Let $\textit{B}^{m}_{r}$ be the ball of radius $r>0$ in $\mathbb{R}^{2m + 1}$. According to Lemma \ref{lemma2.1}, for any $y,z\in \textit{B}^{m}_{r}$, we have
\begin{equation}\label{4.3}
\begin{aligned}
&\|F_{m}y\|\leq \beta r^{3} + (2\alpha +\beta + \beta\gamma)r^{2} + (4\nu + \beta\gamma + \lambda)r +\|f^{m}\| \\
&~~~~~~~~~~\leq  \beta r^{3} + (2\alpha +\beta + \beta\gamma)r^{2} + (4\nu + \beta\gamma + \lambda)r +\|f\|=M_r,
\end{aligned}
\end{equation}
\begin{equation}\label{4.4}
\|F_{m}y - F_{m}z\| \leq L_{r}\|y-z\|,\hspace{5.4cm}
\end{equation}
here $M_{r}$ and $L_{r}$ are the same as those in Lemma \ref{lemma2.1}. By using inequality \eqref{4.3} and \eqref{4.4}, the existence of solution and attractor for the truncated equation \eqref{4.2} can be proven.

\begin{theorem}\label{theorem4.1}
For any $\epsilon\in(0,\epsilon^{*}]$ and $u_{0}^{\epsilon,m}\in \textit{B}^m_{r^{*}}$ with $m\in \mathbb{N}$, the truncated system \eqref{4.2} possesses a unique solution $u_{n}^{\epsilon,m}\in \textit{B}^m_{r^{*}}$ for any $n\in \mathds{N}$. Furthermore, the solution semigroup possesses a unique finite-dimensional numerical attractor $\mathcal{A}^{\epsilon}_{m}$ (also known as the truncated numerical attractor) as given below:
\begin{equation}\label{theorem4.1-1}
\mathcal{A}^{\epsilon}_{m}=\bigcap_{N=1}^{\infty}\overline{\bigcup_{n=N}^{\infty}u_n^{\epsilon,m}\left(\textit{B}^m_{r^{*}}\right)}.
\end{equation}
\end{theorem}
\begin{proof}
Using the same method as in Theorem \ref{theorem2.1}, we have
\begin{equation*}
\begin{aligned}
\|u_{n}^{\epsilon,m}\|^2 &\leq \frac{1}{1+\epsilon (\lambda-\lambda^{*})} \left( \|u_{n-1}^{\epsilon,m}\|^{2} +\frac{\epsilon}{(\lambda-\lambda^{*})} \|f|^{m}\|^{2}\right)\\
&\leq \frac{1}{1+\epsilon (\lambda-\lambda^{*})} \left( \|u_{n-1}^{\epsilon,m}\|^{2} +\frac{\epsilon}{(\lambda-\lambda^{*})} \|f\|^{2}\right).
\end{aligned}
\end{equation*}
According to Theorem \ref{theorem2.1}, it is known that the truncated system \eqref{4.2} has a unique solution $u_{n}^{\epsilon,m}\left(u_{0}^{\epsilon,m}\right)\in \textit{B}^m_{r^{*}}$ for any $\epsilon\in(0,\epsilon^{*}]$ and $u_{0}^{\epsilon,m}\in\textit{B}^m_{r^{*}}$.  Furthermore, the solution semigroup on $\textit{B}^m_{r^{*}}$ is compact and absorbing in the finite-dimensional space. Thus, the truncated system \eqref{4.2} possesses a unique attractor $\mathcal{A}^{\epsilon}_{m}$ as defined in \eqref{theorem4.1-1}.
\end{proof}

\subsection{Convergence from the truncated numerical attractor to the numerical attractor}
In this subsection, we will prove the convergence of truncated numerical attractor $\mathcal{A}^{\epsilon}_{m}$ as $m\rightarrow\infty$. First, we will prove that the tail of any element in $\mathcal{A}^{\epsilon}_{m}$ becomes uniformly small as $m\rightarrow\infty$.

\begin{lemma}\label{lemma4.1}
Suppose $\epsilon\in(0,\epsilon^{*}]$, for any $\delta>0$, there exists $I_{\delta}\in \mathbb{N}$ such that
\begin{equation}\label{lemma4.1-1}
\sum\limits_{I_{\delta} \leq |i| \leq m}\left| z_{i} \right|^{2} < \delta,\mspace{6mu}\mspace{6mu}\forall z = \left( z_{i} \right)_{|i| \leq m} \in \mathcal{A}_{m}^{\epsilon},~~~\forall m\geq I_{\delta},~m\in \mathbb{N}.
\end{equation}
\end{lemma}
\begin{proof}
Let $\xi_{k}$ be the cut-off function defined in Theorem \ref{theorem3.1}, and define $\xi^m_{k}=\left(\xi_{k,i}\right)_{|i|\leq m}$ as the fintie-dimensional cut-off function. Taking the inner product of \eqref{4.2} with $\xi^m_{k}u_{n}^{\epsilon,m}$ in $\mathbb R^{2m+1}$, we have 
\begin{equation}\label{lemma4.1-2}
\sum_{|i|\leq m}\xi_{k,i}|u_{n,i}^{\epsilon,m}|^{2}=\left(u_{n-1}^{\epsilon,m}, \xi^m_{k}u_{n}^{\epsilon,m}\right) + \epsilon \left( \nu \Lambda_{m} u_{n}^{\epsilon,m}
-\alpha u_{n}^{\epsilon,m} D^{-}_{m} u_{n}^{\epsilon,m} + \beta u_{n}^{\epsilon,m}(1-u_{n}^{\epsilon,m})(u_{n}^{\epsilon,m}-\gamma) -\lambda u_{n}^{\epsilon,m}+f^{m}, \xi^m_{k}u_{n}^{\epsilon,m} \right).
\end{equation}
According to Theorem \ref{theorem4.1}, we know $|u_{n,i}^{\epsilon,m}| \leq \|u_{n}^{\epsilon,m}\| \leq r^{*}$, then
\begin{equation}\label{lemma4.1-3}
\begin{aligned}
\epsilon \nu \left( \Lambda_{m} u_{n}^{\epsilon,m}, \xi_{k}^m u_{n}^{\epsilon,m}\right)
\leq & \frac{5}{2} \epsilon \nu \xi_{k,-m}|u_{n,-m}^{\epsilon,m}|^{2} + \frac{1}{2} \epsilon \nu \xi_{k,-m}|u_{n,-m+1}^{\epsilon,m}|^{2} \\
& + \frac{1}{2}\epsilon \nu \sum_{-m+1 \leq i \leq m-1} \xi_{k,i} 
\left[ |u_{n,i-1}^{\epsilon,m}|^{2} 
+ 6|u_{n,i}^{\epsilon,m}|^{2}+|u_{n,i+1}^{\epsilon,m}|^{2} \right] \\
& + \frac{1}{2} \epsilon \nu \xi_{k,m} |u_{n,m-1}^{\epsilon,m}|^{2} + \frac{3}{2} \epsilon \nu \xi_{k,m}|u_{n,m}^{\epsilon,m}|^{2}.
\end{aligned}
\end{equation}
When $i=-m$, the right-hand side of the inequality \eqref{lemma4.1-3} is as follows:
\begin{equation*}
\begin{aligned}
&\frac{5}{2} \epsilon \nu \xi_{k,-m}|u_{n,-m}^{\epsilon,m}|^{2} + \frac{1}{2} \epsilon \nu \xi_{k,-m+1}|u_{n,-m}^{\epsilon,m}|^{2}\\
= & \epsilon \nu \left( \frac{5}{2} \xi_{k,-m} + \frac{1}{2} \xi_{k,-m+1}\right) |u_{n,-m}^{\epsilon,m}|^{2}\\
\leq & \epsilon \nu \left(3 \xi_{k,-m} + \frac{1}{2} \xi_{k,-m+1} - \frac{1}{2} \xi_{k,-m}\right) |u_{n,-m}^{\epsilon,m}|^{2}\\
\leq & 3\epsilon \nu \xi_{k,-m} |u_{n,-m}^{\epsilon,m}|^{2} + \frac{c_1}{2k}\epsilon \nu |u_{n,-m}^{\epsilon,m}|^{2}.
\end{aligned}
\end{equation*}
When $i=-m+1$, the right-hand side of the inequality \eqref{lemma4.1-3} is expressed as:
\begin{equation*}
\begin{aligned}
&\frac{1}{2} \epsilon \nu \xi_{k,-m}|u_{n,-m+1}^{\epsilon,m}|^{2} + \frac{1}{2} \epsilon \nu \xi_{k,-m+2}|u_{n,-m+1}^{\epsilon,m}|^{2} + 3 \epsilon \nu \xi_{k,-m+1}|u_{n,-m+1}^{\epsilon,m}|^{2}\\
= & \frac{1}{2}\epsilon \nu \left( \xi_{k,-m} + 6 \xi_{k,-m+1} + \xi_{k,-m+2}\right) |u_{n,-m+1}^{\epsilon,m}|^{2}\\
= & \frac{1}{2}\epsilon \nu \left( 8 \xi_{k,-m+1} + \xi_{k,-m} - \xi_{k,-m+1} + \xi_{k,-m+2} - \xi_{k,-m+1}\right) |u_{n,-m+1}^{\epsilon,m}|^{2}\\
\leq & 4\epsilon\nu \xi_{k,-m+1} |u_{n,-m+1}^{\epsilon,m}|^{2} + \frac{c_1}{k}\epsilon\nu |u_{n,-m+1}^{\epsilon,m}|^{2}.
\end{aligned}
\end{equation*}
When $i=-m+2,\cdots,m-2$, the right-hand side of the inequality \eqref{lemma4.1-3} is presented as:
\begin{equation*}
\begin{aligned}
&\frac{1}{2} \epsilon \nu \xi_{k,i+1}|u_{n,i}^{\epsilon,m}|^{2} + \frac{6}{2} \epsilon \nu \xi_{k,i}|u_{n,i}^{\epsilon,m}|^{2} + \frac{1}{2} \epsilon \nu \xi_{k,i-1}|u_{n,i}^{\epsilon,m}|^{2}\\
= & \frac{1}{2} \epsilon \nu \left( 8 \xi_{k,i} + \xi_{k,i+1} -\xi_{k,i} + \xi_{k,i-1}-\xi_{k,i} \right) |u_{n,i}^{\epsilon,m}|^{2}\\
= & 4\epsilon\nu \xi_{k,i}|u_{n,i}^{\epsilon,m}|^{2} + \frac{c_1}{k}\epsilon\nu|u_{n,i}^{\epsilon,m}|^{2}.
\end{aligned}
\end{equation*}
When $i=m-1$, the right-hand side of the inequality \eqref{lemma4.1-3} is shown below:
\begin{equation*}
\begin{aligned}
&\frac{6}{2} \epsilon \nu \xi_{k,m-1}|u_{n,m-1}^{\epsilon,m}|^{2} + \frac{1}{2} \epsilon \nu \xi_{k,m-2}|u_{n,m-1}^{\epsilon,m}|^{2} + \frac{1}{2} \epsilon \nu \xi_{k,m}|u_{n,m-1}^{\epsilon,m}|^{2}\\
= & \frac{1}{2}\epsilon \nu \left( 8\xi_{k,m-1} + \xi_{k,m-2} - \xi_{k,m-1} + \xi_{k,m} - \xi_{k,m-1}\right) |u_{n,m-1}^{\epsilon,m}|^{2}\\
\leq & 4\epsilon \nu \xi_{k,m-1} |u_{n,m-1}^{\epsilon,m}|^{2} + \frac{c_1}{k}\epsilon \nu |u_{n,m-1}^{\epsilon,m}|^{2}.
\end{aligned}
\end{equation*}
When $i=m$, the right-hand side of the inequality \eqref{lemma4.1-3} is stated as:
\begin{equation*}
\frac{1}{2} \epsilon \nu \xi_{k,m-1}|u_{n,m}^{\epsilon,m}|^{2} + \frac{3}{2} \epsilon \nu \xi_{k,m}|u_{n,m}^{\epsilon,m}|^{2} = \frac{1}{2}\epsilon \nu \left( 4\xi_{k,m} + \xi_{k,m-1} - \xi_{k,m} \right) |u_{n,m}^{\epsilon,m}|^{2}
\leq 2\epsilon\nu \xi_{k,m} |u_{n,m}^{\epsilon,m}|^{2} + \frac{c_1}{2k}\epsilon\nu |u_{n,m}^{\epsilon,m}|^{2}.
\end{equation*}
Based on the situation of each component, we have uniformly that
\begin{equation*}
\begin{aligned}
\epsilon \nu \left( \Lambda_{m} u_{n}^{\epsilon,m}, \xi_{k}^m u_{n}^{\epsilon,m}\right)
\leq & 4\epsilon \nu \sum_{|i|\leq m} \xi_{k,i}|u_{n,i}^{\epsilon,m}|^{2}
+ \frac{c_1}{k}\epsilon\nu \sum_{|i|\leq m} |u_{n,i}^{\epsilon,m}|^{2}\\
\leq & 4\epsilon \nu \sum_{|i|\leq m} \xi_{k,i}|u_{n,i}^{\epsilon,m}|^{2}
+ \frac{c_1}{k}\epsilon\nu \left( r^{*}\right)^{2}.
\end{aligned}
\end{equation*}
Using the same method, the following inner product in \eqref{lemma4.1-2} is estimated as follows:
\begin{equation*}
\begin{aligned}
&-\epsilon \alpha \left( u_{n}^{\epsilon,m} D^{-}_{m} u_{n}^{\epsilon,m}, \xi_{k}u_{n}^{\epsilon,m} \right)
= -\epsilon \alpha \sum_{|i|\leq m} \xi_{k,i}|u_{n,i}^{\epsilon,m}|^{2} \left( D^{-}_{m} u_{n}^{\epsilon,m}\right)_{i}\\
= & \epsilon \alpha \sum_{-m+1 \leq i \leq m} \xi_{k,i} \left( |u_{n,i}^{\epsilon,m}|^{2} u_{n,i-1}^{\epsilon,m} - |u_{n,i}^{\epsilon,m}|^{2} u_{n,i}^{\epsilon,m} \right) + \epsilon\alpha \xi_{k,-m} |u_{n,-m}^{\epsilon,m}|^{2} u_{n,-m}^{\epsilon,m}\\
\leq &\epsilon \alpha \sum_{-m+1 \leq i \leq m} \xi_{k,i} \left( \frac{2}{3}|u_{n,i}^{\epsilon,m}|^{3} + \frac{1}{3} |u_{n,i-1}^{\epsilon,m}|^{3} + |u_{n,i}^{\epsilon,m}|^3 \right) + \epsilon\alpha \xi_{k,-m} |u_{n,-m}^{\epsilon,m}|^{3}\\
= &\frac{5}{3}\epsilon \alpha \sum_{-m+1 \leq i \leq m} \xi_{k,i} |u_{n,i}^{\epsilon,m}|^{3} + \frac{1}{3}\epsilon \alpha \sum_{-m+1 \leq i \leq m} \left( \xi_{k,i} - \xi_{k,i-1} + \xi_{k,i-1}\right)|u_{n,i-1}^{\epsilon,m}|^{3} + \epsilon \alpha \xi_{k,-m}|u_{n,-m}^{\epsilon,m}|^{3}\\
\leq & 2\epsilon \alpha \sum_{|i| \leq m} \xi_{k,i}|u_{n,i}^{\epsilon,m}|^{3}+ \frac{c_1}{3k}\epsilon\alpha \sum_{-m+1 \leq i \leq m}|u_{n,i-1}^{\epsilon,m}|^{3}\\
\leq & 2\epsilon \alpha \sum_{|i| \leq m} \xi_{k,i}|u_{n,i}^{\epsilon,m}|^{3}+ \frac{c_1}{3k}\epsilon\alpha \left(r^{*}\right)^{3}.
\end{aligned}
\end{equation*}
Using Young's inequality, the first and the last inner products in \eqref{lemma4.1-2} are estimated as follows:
\begin{equation*}
\left(u_{n-1}^{\epsilon,m}, \xi^m_{k}u_{n}^{\epsilon,m}\right) \leq \frac{1}{2} \sum_{|i|\leq m} \xi_{k,i}\left(u_{n-1,i}^{\epsilon,m}\right)^{2} + \frac{1}{2} \sum_{|i|\leq m} \xi_{k,i}\left(u_{n,i}^{\epsilon,m}\right)^{2},\hspace{1.5cm}
\end{equation*}
\begin{equation*}
\epsilon \left( f^{m}, \xi_{k}^m u_{n}^{\epsilon,m} \right) \leq 
\frac{\epsilon(\lambda-\lambda^{*})}{2}\sum_{|i|\leq m} \xi_{k,i}\left(u_{n,i}^{\epsilon,m}\right)^{2} + \frac{\epsilon}{2(\lambda-\lambda^{*})}\sum_{|i|\leq m} \xi_{k,i} f_{i}^{2}. 
\end{equation*}
The two remaining inner products in \eqref{lemma4.1-2} are represented in component form
\begin{equation*}
\epsilon \beta \left( u_{n}^{\epsilon,m}(1-u_{n}^{\epsilon,m})(u_{n}^{\epsilon,m}-\gamma),  \xi_{k}^m u_{n}^{\epsilon,m} \right) 
=  - \epsilon \beta \sum_{|i|\leq m} \xi_{k,i}\left(u_{n,i}^{\epsilon,m}\right)^{4} + \epsilon \beta (1+\gamma) \sum_{|i|\leq m} \xi_{k,i}\left(u_{n,i}^{\epsilon,m}\right)^{3} - \epsilon \beta \gamma \sum_{|i|\leq m} \xi_{k,i}\left(u_{n,i}^{\epsilon,m}\right)^{2},
\end{equation*}
\begin{equation*}
 - \epsilon \lambda \left( u_{n}^{\epsilon,m}, \xi_{k}^m u_{n}^{\epsilon,m} \right) =  - \epsilon \lambda \sum_{|i|\leq m} \xi_{k,i}\left(u_{n,i}^{\epsilon,m}\right)^{2}.  \hspace{4.5cm}
\end{equation*}
Based on the above derivation, we have the following estimate for \eqref{lemma4.1-2} according to the notation in \eqref{2.1-12}
\begin{equation*}
\begin{aligned}
&\sum_{|i|\leq m}\xi_{k,i}|u_{n,i}^{\epsilon,m}|^{2} \leq \frac{1}{2} \sum_{|i|\leq m} \xi_{k,i} |u_{n-1,i}^{\epsilon,m}|^{2} + \frac{1}{2} \sum_{|i|\leq m} \xi_{k,i} |u_{n,i}^{\epsilon,m}|^{2} 
+  4\epsilon \nu \sum_{|i|\leq m} \xi_{k,i}|u_{n,i}^{\epsilon,m}|^{2}
+ \frac{c_1}{k}\epsilon\nu \left( r^{*}\right)^{2} + \frac{c_1}{3k}\epsilon\alpha \left(r^{*}\right)^{3}\\
&\hspace{2.3cm} + \epsilon \left( 2\alpha + \beta + \beta\gamma \right) \sum_{|i| \leq m} \xi_{k,i} |u_{n,i}^{\epsilon,m}|^{3} - \epsilon\beta \sum_{|i| \leq m} \xi_{k,i} |u_{n,i}^{\epsilon,m}|^{4} - \epsilon \beta\gamma \sum_{|i| \leq m} \xi_{k,i} |u_{n,i}^{\epsilon,m}|^{2} \\
&\hspace{2.3cm} - \epsilon \lambda \sum_{|i| \leq m} \xi_{k,i} |u_{n,i}^{\epsilon,m}|^{2}  + \frac{\epsilon(\lambda-\lambda^{*})}{2} \sum_{|i| \leq m} \xi_{k,i} |u_{n,i}^{\epsilon,m}|^{2} + \frac{\epsilon}{2(\lambda-\lambda^{*})} \sum_{|i| \leq m} \xi_{k,i} f_{i}^{2}\\
&\hspace{2cm} \leq \frac{1}{2} \sum_{|i|\leq m} \xi_{k,i} |u_{n-1,i}^{\epsilon,m}|^{2} 
+ \frac{\epsilon}{2(\lambda-\lambda^{*})} \sum_{|i| \leq m} \xi_{k,i} f_{i}^{2} + \epsilon\frac{c_3}{2k}\\
&\hspace{2.3cm}  + \left[ \frac{1}{2} + 4\epsilon \nu - \epsilon\beta\gamma - \epsilon\lambda +  \frac{\epsilon(\lambda-\lambda^{*})}{2} 
+\epsilon\frac{(2\alpha + \beta + \beta\gamma)^{2}}{4\beta} \right] \sum_{|i|\leq m} \xi_{k,i} |u_{n,i}^{\epsilon,m}|^{2} \\
& \hspace{2cm}= \frac{1}{2} \sum_{|i|\leq m} \xi_{k,i} |u_{n-1,i}^{\epsilon,m}|^{2}+ \frac{1-\epsilon(\lambda-\lambda^{*})}{2}\sum_{|i|\leq m} \xi_{k,i} |u_{n,i}^{\epsilon,m}|^{2} + \epsilon \left[ \frac{c_3}{2k} + \frac{1}{2(\lambda-\lambda^{*})} \sum_{|i| \geq k} f_{i}^{2} \right],
\end{aligned}
\end{equation*}
here $c_3=2c_1 \nu\left( r^{*}\right)^2+\frac{2}{3}c_1 \alpha\left( r^{*}\right)^3$. The above inequality is transformed into 
\begin{equation*}
\sum_{|i|\leq m}\xi_{k,i}|u_{n,i}^{\epsilon,m}|^{2} \leq \frac{1}{1 + \epsilon(\lambda - \lambda^{*})} \left[ \sum_{|i|\leq m} \xi_{k,i} |u_{n-1,i}^{\epsilon,m}|^{2} 
 + \epsilon\left(\frac{c_3}{k} + \frac{1}{\lambda-\lambda^{*}} \sum_{|i| \geq k} f_{i}^{2}\right)\right]. 
\end{equation*}
Thus, there exists a constant $k_{\delta}>0$ such that
\begin{equation*}
\frac{c_3}{k} + \frac{1}{\lambda-\lambda^{*}} \sum_{|i| \geq k} f_{i}^{2} < \frac{\delta(\lambda-\lambda^{*})}{2},~~~\forall k\geq k_{\delta},    \hspace{1cm}
\end{equation*}
and then
\begin{equation}\label{lemma4.1-4}
\sum_{|i|\leq m}\xi_{k,i}|u_{n,i}^{\epsilon,m}|^{2} \leq \frac{1}{1 + \epsilon(\lambda - \lambda^{*})} \sum_{|i|\leq m} \xi_{k,i} |u_{n-1,i}^{\epsilon,m}|^{2} + \frac{\epsilon (\lambda-\lambda^{*})}{1+(\lambda-\lambda^{*})} \frac{\delta}{2},~~~\forall k\geq k_{\delta}.
\end{equation}
By recursively deriving the inequality \eqref{lemma4.1-4},
for any $k\geq k_{\delta}$ and $u_{0}^{\epsilon,m}\in \textit{B}^m_{r^{*}}$, we have
\begin{equation*}
\sum_{|i|\leq m}\xi_{k,i}|u_{n,i}^{\epsilon,m}|^{2} \leq \frac{1}{\left[1 + \epsilon(\lambda - \lambda^{*})\right]^{n}} \sum_{|i|\leq m} \xi_{k,i} |u_{0,i}^{\epsilon,m}|^{2} + \frac{\delta}{2}\leq \frac{\left(r^{*}\right)^{2}}{\left[1 + \epsilon(\lambda - \lambda^{*})\right]^{n}} + \frac{\delta}{2}.
\end{equation*}
Hence, there exists $N_{\delta}$, we obtain as follows:
\begin{equation*}
\sum_{|i|\leq m} \xi_{k,i} |u_{n,i}^{\epsilon,m}|^{2}<\delta,~~~\forall k\geq k_{\delta}~ and~ n\geq N_{\delta}.
\end{equation*}
Specially, for $I_{\delta}=2k_{\delta}$, we get 
\begin{equation*}
\sum_{I_{\delta}\leq |i|\leq m} |u_{N_{\delta},i}^{\epsilon,m}(y)|^{2} < \delta, ~~~\forall y\in \textit{B}^m_{r^{*}}.
\end{equation*}
For any $\epsilon\in(0,\epsilon^{*}]$ and $m\in \mathds{N}$, suppose $z\in \mathcal{A}_{m}^{\epsilon}$. Then for any $y\in \mathcal{A}_{m}^{\epsilon}\subset \textit{B}^m_{r^{*}}$, by the invariance property, we obtain $z=u_{N_{\delta}}^{\epsilon,m}(y)$. Therefore
\begin{equation*}
\sum_{I_{\delta}\leq |i|\leq m} z_{i}^{2} = \sum_{I_{\delta}\leq |i|\leq m}|u_{N_{\delta},i}^{\epsilon,m}(y)|^{2} < \delta.
\end{equation*}
The lemma is proven.
\end{proof}

The next lemma follows from \cite{Temam-1997}.
\begin{lemma}\label{lemma4.2}
$z\in \mathcal{A}^{\epsilon}_{m}$ if and only if there is a bounded solution $u^{m}_{n}$ 
of the truncated system \eqref{4.2}
such that $u_{0}^{m}=z$, while $y\in \mathcal{A}^{\epsilon}$ if and only if there is a bounded solution $u_{n}$ of the IES \eqref{2.2} in $\ell^2$ such that $u_{0}=y$.
\end{lemma}

To prove the convergence result of $\mathcal{A}^{\epsilon}_{m}$ under the Hausdorff semi-distance, observe that any element in $\mathbb R^{2m+1}$ can be naturally extended to an element in $\ell^2$. The null-expansion of a point $z\in \mathbb R^{2m+1}$ can be defined as follows:
\begin{equation*}
\tilde{z}_{i}=0,~~~\forall~ |i|>m;~~\tilde{z}_{i}=z_{i},~~\forall |i|\leq m.
\end{equation*}
From this viewpoint, the truncated attractor $\mathcal{A}_{m}^{\epsilon}$ naturally extends to a null-expansion $\widetilde{\mathcal{A}_{m}^{\epsilon}}\subset \ell^2$.

\begin{theorem}\label{theorem4.2}
Suppose $\epsilon\in(0,\epsilon^{*}]$. As $m\rightarrow \infty$, the truncated attractor $\mathcal{A}^{\epsilon}_{m}$ in \eqref{theorem4.1-1} is  upper semi-convergence to the numerical attractor
$\mathcal{A}^{\epsilon}$ of the IES \eqref{2.2} under the Hausdorff semi-distance, i.e., 
\begin{equation}\label{theorem4.2-1}
\lim\limits_{m\rightarrow\infty}\text{dist}_{\ell^2}\left( \mathcal{A}_{m}^{\epsilon},\mathcal{A}^{\epsilon} \right) = \lim\limits_{m\rightarrow\infty}\text{dist}_{\ell^2}\left( \widetilde{\mathcal{A}_{m}^{\epsilon}},\mathcal{A}^{\epsilon} \right) = 0.
\end{equation}
\end{theorem}
\begin{proof}  
We prove this theorem by contradiction. Assuming that \eqref{theorem4.2-1} is false, there exist a constant $\eta_{0}>0$, a subsequence $\{m_{k}\}$ and $z^{m_{k}}\in \mathcal{A}^{\epsilon}_{m_{k}}$ such that
\begin{equation}\label{theorem4.2-2}
\text{d}\left(z^{m_k}, \mathcal{A}^{\epsilon}\right) = \text{d}\left(\widetilde{z^{m_k}}, \mathcal{A}^{\epsilon}\right)\geq \eta_{0},~~~\forall k\in \mathds{N}.
\end{equation}
Because $z^{m_{k}}\in \mathcal{A}^{\epsilon}_{m_{k}}$, the unique solution $u_{n}^{\epsilon,m_{k}}=u_{n}^{\epsilon,m_{k}}(z^{m_{k}})$ belongs to $\mathcal{A}^{\epsilon}_{m_{k}}$ by Lemma \ref{lemma4.2}. The null-expansion of $u_{n}^{\epsilon,m_{k}}$ is denoted by $\widetilde{u_{n}^{\epsilon,m_k}}$. By Lemma \ref{lemma4.1}, for any $\delta>0$, there exists $I_{\delta}\in \mathds{N}$ such that
\begin{equation*}
\sum\limits_{|i| \geq I(\delta)}\left| \widetilde{u_{n,i}^{\epsilon,m_k}} \right|^{2} 
=\sum\limits_{I(\delta) \leq |i| \leq m_k}\left| u_{n,i}^{\epsilon,m_{k}}\right|^{2} < \delta
,~~~\forall n\in \mathds{Z},~k\in \mathbb{N}.
\end{equation*}
According to Theorem \ref{theorem4.1}, the attractor is contained within a ball of radius $r^{*}$. Therefore, the unique solution sequence $\left(u_{n,i}^{\epsilon,m_{k}}\right)_{|i|<I_{\delta}}$ is bounded in $\mathbb R^{2I_{\delta}-1}$ and the null-expansion $\widetilde{u_{n}^{\epsilon,m_k}}$ is relatively compact in $\ell^2$. There exists a subsequence of $\{k\}$ (for convenience, we still denote it as $k$) and $u_{n}^{*}\in \ell^2$ such that 
\begin{equation}\label{theorem4.2-3}
\lim\limits_{k\rightarrow\infty}\|\widetilde{u_{n}^{\epsilon,m_k}}- u_{n}^{*}\| = 0,~~~\forall n\in \mathds{Z}.
\end{equation}
In the following, we will prove that $u_{n}^{*}$ is the solution of the IES \eqref{2.2}. Because $u_{n}^{\epsilon,m_{k}}$ is the solution of the truncated system \eqref{4.2}, for any $ n\in \mathds{Z}$ and $k\in \mathds{N}$ we have 
\begin{equation*}
u_{n}^{\epsilon,m_{k}} = u_{n-1}^{\epsilon,m_{k}} + \epsilon \left( \nu \Lambda_{m_{k}} u_{n}^{\epsilon,m_{k}}
-\alpha u_{n}^{\epsilon,m_{k}} D^{-}_{m_{k}} u_{n}^{\epsilon,m_{k}} + \beta u_{n}^{\epsilon,m_{k}}(1-u_{n}^{\epsilon,m_{k}})(u_{n}^{\epsilon,m_{k}}-\gamma) -\lambda u_{n}^{\epsilon,m_{k}}+f^{m_{k}} \right).
\end{equation*}
According to the definition of $m_{k}$-truncation operator, for each component $i\in \mathds{Z}$ as $m_{k}\rightarrow \infty$, we have
\begin{equation*}
\begin{aligned}
&\left( \Lambda_{m_{k}} u_{n}^{\epsilon,m_{k}} \right)_{i} = \left( \Lambda \widetilde{u_{n}^{\epsilon,m_{k}}} \right)_{i},~~~~\left( u_{n}^{\epsilon,m_{k}} D^{-}_{m_{k}} u_{n}^{\epsilon,m_{k}} \right)_{i} = \widetilde{u}_{n,i}^{\epsilon,m_{k}} \left(D^{-} \widetilde{u_{n}^{\epsilon,m_{k}}} \right)_{i},\\
&   u_{n,i}^{\epsilon,m_{k}}(1-u_{n,i}^{\epsilon,m_{k}})(u_{n,i}^{\epsilon,m_{k}}-\gamma) =  \widetilde{u}_{n,i}^{\epsilon,m_{k}}(1-\widetilde{u}_{n,i}^{\epsilon,m_{k}})(\widetilde{u}_{n,i}^{\epsilon,m_{k}}-\gamma) ,
~~~~u_{n,i}^{\epsilon,m_{k}} = \widetilde{u}_{n,i}^{\epsilon,m_{k}},
~~~~\left( f^{m_{k}} \right)_{i} = f_{i}, ~~~\forall n\in \mathds{Z},
\end{aligned}
\end{equation*}
here $ \widetilde{u}_{n,i}^{\epsilon,m_{k}} $ represents the $i$-component of the null-expansion $\widetilde{u_{n}^{\epsilon,m_{k}}}$. As a result, it satisfies
\begin{equation}\label{theorem4.2-4}
\widetilde{u}_{n,i}^{\epsilon,m_{k}} = \widetilde{u}_{n-1,i}^{\epsilon,m_{k}} + \epsilon \left[ \nu \left(\Lambda \widetilde{u_{n}^{\epsilon,m_{k}}}\right)_{i}
-\alpha \widetilde{u}_{n,i}^{\epsilon,m_{k}}\left( 
 D^{-} \widetilde{u_{n}^{\epsilon,m_{k}}} \right)_{i}+ \beta \widetilde{u}_{n,i}^{\epsilon,m_{k}}(1-\widetilde{u}_{n,i}^{\epsilon,m_{k}})(\widetilde{u}_{n,i}^{\epsilon,m_{k}}-\gamma) -\lambda \widetilde{u}_{n,i}^{\epsilon,m_{k}}+f_{i} \right].
\end{equation}
As $k\rightarrow \infty$ (with $m_{k}\rightarrow\infty$ as well), from \eqref{theorem4.2-3} we have
\begin{equation*}
u_{n,i}^{*} = u_{n-1,i}^{*} + \epsilon \left[ \nu \left(\Lambda u_{n}^{*}\right)_{i}
-\alpha u_{n,i}^{*} \left(D^{-} u_{n}^{*} \right)_{i}+ \beta u_{n,i}^{*} (1-u_{n,i}^{*} )(u_{n,i}^{*} -\gamma) -\lambda u_{n,i}^{*} + f_{i}\right], ~~~\forall i,~n\in \mathds{Z}.
\end{equation*}
This indicates that $u_{n}^{*}$ is the solution of the IES \eqref{2.2}. According to Theorem \ref{4.1}, it follows that $\widetilde{u_{n}^{\epsilon,m_{k}}} \in  \textit{B}_{r^{*}}$. From \eqref{theorem4.2-3}, we get
\begin{equation*}
\|u_{n}^{*}\| = \lim_{k\rightarrow \infty}\| \widetilde{u_{n}^{\epsilon,m_{k}}} \|\leq r^{*},~~~\forall n\in \mathds{Z},
\end{equation*}
which implies that the solution $u_{n}^{*}$ of the IES \eqref{2.2} is bounded. According to Lemma \ref{lemma4.2}, It follows that $u_{0}^{*}\in \mathcal{A}^{\epsilon}$ and
\begin{equation*}
\lim_{k\rightarrow \infty} \widetilde{z^{m_{k}}} = \lim_{k\rightarrow \infty}  \widetilde{u_{0}^{\epsilon,m_{k}}(z^{m_{k}})} = u_{0}^{*}.
\end{equation*}
This contradicts the assumption \eqref{theorem4.2-2}, and the theorem is proved.  
\end{proof}

\subsection{Bounds and continuity properties of numerical attractors}
In this part, We will derive the bounds and continuity properties of the numerical attractor $\mathcal{A}^{\epsilon}$ and the truncated attractor $\mathcal{A}^{\epsilon}_{m}$ with respect to the Hausdorff semi-distance and norm:
\begin{equation*}
\rho(A,B)=\max \{d(A,B),~d(B,A)\},~~\|A\|=\rho(A,\{0\}),~~\forall A,~B\subset \ell^2~\text{or}~\mathbb R^{2m+1}.
\end{equation*}
Two attractors will be represented as $\mathcal{A}^{\epsilon}(f, \lambda)$ and $\mathcal{A}_{m}^{\epsilon}(f, \lambda)$, which depend on the external force $f \in \ell^2$ and the damping constant $\lambda > \lambda^{*}$.
\begin{theorem}\label{theorem4.3}
(i) For any $\epsilon\in(0,\epsilon^{*}]$ and $m\in \mathds{N}$, the following upper bounds hold with respect to the two attractors:
\begin{equation}\label{theorem4.3-1}
\|\mathcal{A}^{\epsilon}(f, \lambda)\|\leq \frac{\|f\|}{\lambda-\lambda^{*}},~~~\|\mathcal{A}_{m}^{\epsilon}(f, \lambda)\|\leq \frac{\|f^{m}\|}{\lambda-\lambda^{*}}.
\end{equation}
In particular, if $f=0$, both attractors are reduced to a point, i.e.,
\begin{equation}\label{theorem4.3-3}
\mathcal{A}^{\epsilon}(0, \lambda) = \mathcal{A}_{m}^{\epsilon}(0, \lambda) = \{0\}.\hspace{1.5cm}
\end{equation}
(ii) The following continuity properties hold with respect to the two attractors:
\begin{equation}\label{theorem4.3-2}
\begin{aligned}
&\lim_{f\rightarrow 0}\rho\left(\mathcal{A}^{\epsilon}(f, \lambda),\{0\}\right) = 0,~~~ \lim_{f\rightarrow 0}\rho\left(\mathcal{A}_{m}^{\epsilon}(f, \lambda),\{0\}\right) = 0,\\
&\lim_{\lambda\rightarrow +\infty}\rho\left(\mathcal{A}^{\epsilon}(f, \lambda),\{0\}\right) = 0,~~~ \lim_{\lambda\rightarrow +\infty}\rho\left(\mathcal{A}_{m}^{\epsilon}(f, \lambda),\{0\}\right) = 0.
\end{aligned}
\end{equation}
\end{theorem}
\begin{proof}
(i) We prove the second inequality in \eqref{theorem4.3-1}, and the first inequality follows using a similar method. Let $r_0$ be a radius that satisfies
\begin{equation*}
\frac{\|f^{m}\|}{\lambda-\lambda^{*}}<r_{0}\leq 1+\frac{\|f^{m}\|}{\lambda-\lambda^{*}}:=r^{*,m} \leq 1+\frac{\|f\|}{\lambda-\lambda^{*}} = r^{*}.
\end{equation*}
Subsequently, We prove that the ball $\textit{B}^{m}_{r_{0}}$ in $\mathbb R^{2m+1}$ is absorbing for the truncated system \eqref{4.2}. Taking the inner product of \eqref{4.2} with $u_{n}^{\epsilon,m}$, we have
\begin{equation*}
\|u_{n}^{\epsilon,m}\|^{2} = \left(u_{n-1}^{\epsilon,m}, u_{n}^{\epsilon,m}\right) + \epsilon \left( \nu \Lambda_{m} u_{n}^{\epsilon,m}
-\alpha u_{n}^{\epsilon,m} D^{-}_{m} u_{n}^{\epsilon,m} + \beta u_{n}^{\epsilon,m}(1-u_{n}^{\epsilon,m})(u_{n}^{\epsilon,m}-\gamma) -\lambda u_{n}^{\epsilon,m}+f^{m}, u_{n}^{\epsilon,m} \right).
\end{equation*}
Based on the form of matrices $D^{-}_{m}$ and $\Lambda_{m}$ in \eqref{4.1-1}-\eqref{4.1-2}, we get
\begin{equation*}
\begin{aligned}
&\epsilon\nu\left(\Lambda_{m} u_{n}^{\epsilon,m}, u_{n}^{\epsilon,m} \right) \\
=& \epsilon\nu
\sum\limits_{-m+1 \leq i \leq m-1}
\left( -u_{n,i-1}^{\epsilon,m} +2u_{n,i}^{\epsilon,m}-u_{n,i+1}^{\epsilon,m} \right)u_{n,i}^{\epsilon,m} +\epsilon\nu \left( 2u_{n,-m}^{\epsilon,m} - u_{n,-m+1}^{\epsilon,m} \right)u_{n,-m}^{\epsilon,m} +\epsilon\nu \left( -u_{n,m-1}^{\epsilon,m} +u_{n,m}^{\epsilon,m} \right)u_{n,m}^{\epsilon,m} \\
\leq & \epsilon\nu\sum\limits_{-m+1 \leq i \leq m-1} \left[ \frac{1}{2}|u_{n,i-1}^{\epsilon,m}|^2 + 3|u_{n,i}^{\epsilon,m}|^2 + \frac{1}{2}|u_{n,i+1}^{\epsilon,m}|^2 \right] + \frac{5}{2}\epsilon\nu|u_{n,-m}^{\epsilon,m}|^2 + \frac{1}{2}\epsilon\nu|u_{n,-m+1}^{\epsilon,m}|^2 + \frac{1}{2}\epsilon\nu|u_{n,m-1}^{\epsilon,m}|^2 + \frac{3}{2}\epsilon\nu |u_{n,m}^{\epsilon,m}|^2\\
\leq & 4\epsilon\nu \|u_{n}^{\epsilon,m}\|^{2},
\end{aligned}
\end{equation*} 
and 
\begin{equation*}
-\epsilon\alpha\left( u_{n}^{\epsilon,m} D^{-}_{m} u_{n}^{\epsilon,m}, u_{n}^{\epsilon,m} \right) \leq 2 \epsilon\alpha \sum\limits_{-m+1 \leq i \leq m}|u_{n,i}^{\epsilon,m}|^3 + \epsilon\alpha |u_{n,-m}^{\epsilon,m}|^3 \leq 2\epsilon\alpha\|u_{n}^{\epsilon,m}\|_3^3.
\end{equation*}
The inner products below are written in terms of norms
\begin{equation*}
\begin{aligned}
\epsilon\beta \left( u_{n}^{\epsilon,m}(1-u_{n}^{\epsilon,m})(u_{n}^{\epsilon,m}-\gamma), u_{n}^{\epsilon,m}\right) &= \epsilon\beta \sum\limits_{-m \leq i \leq m}\left[ -\left(u_{n,i}^{\epsilon,m} \right)^4 + (1+\gamma) \left( u_{n,i}^{\epsilon,m} \right)^3 - \gamma \left( u_{n,i}^{\epsilon,m} \right)^2 \right]\\
& = -\epsilon\beta \|u_{n}^{\epsilon,m}\|_4^4 + \epsilon\beta(1+\gamma)\|u_{n}^{\epsilon,m}\|_3^3 - \epsilon\beta\gamma \|u_{n}^{\epsilon,m}\|^2,
\end{aligned}
\end{equation*}  
\begin{equation*}
-\epsilon\lambda \left(u_{n}^{\epsilon,m}, u_{n}^{\epsilon,m}\right) = -\epsilon\lambda\|u_{n}^{\epsilon,m}\|^2.\hspace{3.2cm}
\end{equation*}   
For the remainder term, we use Young's inequality
\begin{equation*}
\left(u_{n-1}^{\epsilon,m}, u_{n}^{\epsilon,m}\right) + \epsilon\left(f^{m}, u_{n}^{\epsilon,m}\right) \leq \frac{1}{2}\|u_{n-1}^{\epsilon,m}\|^{2} + \frac{1}{2}\|u_{n}^{\epsilon,m}\|^{2} + \epsilon\frac{(\lambda-\lambda^{*})}{2}\|u_{n}^{\epsilon,m}\|^{2} + \frac{\epsilon}{2(\lambda-\lambda^{*})}\|f^{m}\|^{2}.
\end{equation*}   
Based on all the estimates above, we have
\begin{equation*} 
\|u_{n}^{\epsilon,m}\|^{2} \leq \frac{1}{2}\|u_{n-1}^{\epsilon,m}\|^{2} + \frac{1-\epsilon(\lambda-\lambda^{*})}{2}\|u_{n}^{\epsilon,m}\|^{2} + \frac{\epsilon}{2(\lambda-\lambda^{*})}\|f^{m}\|^{2},
\end{equation*} 
which can be recombined as
\begin{equation*} 
\|u_{n}^{\epsilon,m}\|^{2} \leq \frac{1}{1+\epsilon(\lambda-\lambda^{*})} \|u_{n-1}^{\epsilon,m}\|^{2} + \frac{\epsilon}{[1+\epsilon(\lambda-\lambda^{*})](\lambda-\lambda^{*})}\|f^{m}\|^{2}.
\end{equation*}  
By recursively deriving the above inequality, for any $u_0\in \textit{B}^{m}_{r}$ with $0<r\leq r^{*,m}$, we obtain
\begin{equation*} 
\begin{aligned}
\|u_{n}^{\epsilon,m}(u_0)\|^{2} &\leq \frac{1}{\left[1+\epsilon(\lambda-\lambda^{*})\right]^{n}} \|u_{0}\|^{2} + \frac{\epsilon}{\lambda-\lambda^{*}}\|f^{m}\|^{2}\sum\limits_{j=1}^{n}\frac{1}{\left[1+\epsilon(\lambda-\lambda^{*})\right]^{j}}\\
& \leq \frac{r^2}{\left[1+\epsilon(\lambda-\lambda^{*})\right]^{n}} + \frac{\|f^{m}\|^{2}}{(\lambda-\lambda^{*})^{2}}.
\end{aligned}
\end{equation*} 
Let $\|f^{m}\|/(\lambda-\lambda^{*})<r_0$. For any $r\in (0,r^{*,m}]$, there exists $N=N(r)$ such that for any $n\geq N$ and $u_0\in \textit{B}^{m}_{r}$, we have
\begin{equation*} 
\|u_{n}^{\epsilon,m}(u_0)\|^{2} \leq r^2_0.
\end{equation*}   
Therefore, $\textit{B}^{m}_{r_0}$ is bounded and absorbing for the truncated system \eqref{4.2}.

Because an attractor is contained within any absorbing set, we have
\begin{equation*} 
\mathcal{A}_{m}^{\epsilon}(f, \lambda) \subset \textit{B}^{m}_{r_0}~~\Rightarrow~~\|\mathcal{A}_{m}^{\epsilon}(f, \lambda) \| \leq  r_0,~~~\forall~r_0\in \left( \frac{\|f^{m}\|}{\lambda-\lambda^{*}}, r^{*,m}\right],~~~\forall\epsilon \in (0,\epsilon^{*}].
\end{equation*}
As $r_0\rightarrow \frac{\|f^{m}\|}{\lambda-\lambda^{*}}$, we get
\begin{equation*} 
\|\mathcal{A}_{m}^{\epsilon}(f, \lambda) \| \leq \frac{\|f^{m}\|}{\lambda-\lambda^{*}},~~~\forall \epsilon \in (0,\epsilon^{*}],~m\in \mathds{N}.
\end{equation*} 
The case when $f=0$ and Conclusion (ii) are directly derived from Conclusion (i).
\end{proof}

\section{Finitely dimensional approximation of random attractor}\label{sec:6}
In this section, we firstly provide the existence of the random attractor for the system \eqref{1.1}. Then we establish the upper semi-convergence between random attractor and global attractor. In the end, we derive the
convergence between random and deterministic truncated attractors.

The stochastic Burgers-Huxley lattice system \eqref{1.1} can be expressed in a more general form:
\begin{equation}\label{6.1}
\left\{\begin{array}{l}
d u=\left(\nu\Lambda u-\alpha uD^{-}u+\beta u(1-u)(u-\gamma)-\lambda u+f\right) d t +\epsilon u\circ d W,\\
u(x,0)=u_0(x),
\end{array}\right.
\end{equation}
where $\circ$ denotes the Stratonovich stochastic differential.

\subsection{Existence and upper semi-convergence of random attractor}

Suppose $W(t,\omega)(t\in \mathbb{R})$ represent the standard one-dimensional, two-sided Wiener process with the path $\omega(t)$, defined on the classical Wiener space $
\left( \Omega,\mathcal{F},\mathbb{P},\left\{ \theta_{t} \right\}_{t \in \mathbb{R}} \right)$, where
\begin{equation*}
\Omega = \left\{ \omega \in C\left( \mathbb{R},\mathbb{R} \right):\omega(0) = 0 \right\},
\end{equation*}
and the Borel $\sigma$-algebra $\mathcal{F}$ is generated by the compact-open topology. The shift operator is given by $\theta_{t}\omega( \cdot ) = \omega(t + \cdot ) - \omega( t )$, as detailed in \cite{Bates-2009,Li-2015}. With this framework, we consider the It\^{o} equation $dz + zdt = dW(t)$, which leads to the derivation of the stochastic Ornstein-Uhlenbeck process
\begin{equation*}
z\left( \theta_{t}\omega \right): = -  \int_{- \infty}^{0}e^{ s}\left( \theta_{t}\omega \right)(s)ds
= - \int_{- \infty}^{0}e^{s}\omega(t+s)ds+\omega(t),
\mspace{6mu} t \in \mathbb{R},\mspace{6mu}\omega \in \Omega.
\end{equation*}
Some important properties of the Ornstein-Uhlenbeck process can be found in \cite{Arnold-1998}. There exists a $\theta_{t}-$invariant set $\tilde{\Omega}\subseteq\Omega$ with $\mathbb P(\tilde{\Omega})=1$, for each $\omega\in \tilde{\Omega}$, $z(\theta_{t}\omega)$ is continuous in $t$ and
\begin{equation}\label{6.2}
\lim\limits_{t\rightarrow\pm\infty}\frac{|z(\theta_{t}\omega)|}{|t|}=0
~~\text{and}~~
\lim\limits_{t\rightarrow\pm\infty}\frac{1}{t}\int_{0}^{t}z(\theta_{t}\omega)dt=0.
\end{equation}
In the following parts, we focus solely on the space $\tilde{\Omega}$ instead of $\Omega$, and for simplicity, we will refer to $\tilde{\Omega}$ as $\Omega$.

In order to analyze the random dynamics of SPDEs, we transform the random evolutionary equation to a pathwise PDE with the random parameter. Suppose $u$ be the solution of \eqref{6.1} and set
\begin{equation*}
U^{\epsilon}(t,\tau,\omega,U_{\tau})=e^{-\epsilon z(\theta_{t}\omega)}u(t,\tau,\omega,u_{\tau}),~~~\text{with}~~ U_{\tau}=e^{-\epsilon z(\theta_{t}\omega)}u_{\tau},
\end{equation*}
here $z(\theta_{t}\omega)$ is the Ornstein-Uhlenbeck process. Thus
\begin{equation}\label{6.3}
dU^{\epsilon}=e^{-\epsilon z(\theta_{t}\omega)}du-\epsilon e^{-\epsilon z(\theta_{t}\omega)}u\circ dz(\theta_{t}\omega).
\end{equation}
According to \eqref{6.1} and \eqref{6.3}, we consider the equation with random coefficients as follows:
\begin{equation}\label{6.4}
\begin{aligned}
\frac{dU^{\epsilon}}{dt}=&\nu\Lambda U^{\epsilon}-\alpha e^{\epsilon z(\theta_{t}\omega)} U^{\epsilon}D^{-}U^{\epsilon}-\beta e^{2\epsilon z(\theta_{t}\omega)} \left(U^{\epsilon}\right)^3
+\beta(1+\gamma)e^{\epsilon z(\theta_{t}\omega)}\left(U^{\epsilon}\right)^2\vspace{2.0ex}\\
&
-\beta\gamma U^{\epsilon}-\lambda U^{\epsilon}+e^{-\epsilon z(\theta_{t}\omega)} f +\epsilon U^{\epsilon}z(\theta_{t}\omega),
\end{aligned}
\end{equation}
the initial value condition is 
\begin{equation*}
U_{0}=e^{-\epsilon z(\theta_{t}\omega)}u_{0}.
\end{equation*}
Let $\Phi^{\epsilon}$ denote the continuous random dynamical system related to the problem in 
\eqref{6.4}, which can be expressed as:
\begin{equation*}
\Phi^{\epsilon}\left( \cdot ,\omega,\Phi_{0} \right) = U^{\epsilon}\left( \cdot ,\omega,U_{0} \right) \in C^{1}\left( \lbrack 0, + \infty),\ell^2 \right).
\end{equation*}
A continuous random dynamical system
$\Gamma^{\epsilon}:\mathbb{R}^{+} \times \Omega \times \ell^2 \mapsto \ell^2$ is generated, described as follows:
\begin{equation}\label{6.5}
\Gamma^{\epsilon}(t,\omega)\Phi_{0} = \Phi^{\epsilon}\left( t,\omega,\Phi_{0} \right),\mspace{6mu}\forall\left( t,\omega,\Phi_{0} \right) \in \mathbb{R}^{+} \times \Omega \times \ell^2.
\end{equation}
Below, we present two lemmas regarding $\tilde{\mathcal{D}}$-absorption and tail estimates. 
$\tilde{\mathcal{D}}$ represents the universe of all tempered random sets in $X:=\ell^2$, and $\mathcal{D}\in\tilde{\mathcal{D}}$ if for a.e. $\omega \in \Omega$, the following condition is satisfied
\begin{equation}\label{6.6}
\lim\limits_{t\rightarrow +\infty}e^{-\beta t}\sup\{\| \Phi^{\epsilon}\|_{X} :\Phi^{\epsilon}\in \mathcal{D}\left( \theta_{- t}\omega \right)\} = 0,\mspace{6mu} \forall\beta>0,
\end{equation}
where $\|\cdot\|_X$ denotes the norm in $X$.

\begin{lemma}\label{lemma6.1}
	For any $\omega\in\Omega$, $\tau\in \mathbb{R}$ and $E=\{E(\tau,\omega):\tau\in\mathbb{R},\omega\in\Omega\}\in \mathcal D$, one can find a $T:=T(\tau,\omega,E,\epsilon)>0$ such that for all $t\geq T$ and $\Phi_{0}\in\mathcal{D}(\theta_{-t}\omega)$,
	\begin{equation}\label{6.7}
	\parallel \Phi^{\epsilon}\left( t,\theta_{- t}\omega,\Phi_{0}\right) \parallel_{X}^{2} \leq R(\epsilon,\omega),
	\end{equation}
	here
	\begin{equation}\label{6.8}
	R(\epsilon,\omega) = 1 + \frac{\|f\|^{2}}{\lambda-\lambda^*} \int_{-\infty}^{0}e^{-2\epsilon z(\theta_{s}\omega) -\int_{0}^{s} 2\epsilon z(\theta_{r}\omega)dr
		+(\lambda-\lambda^*) s} ds.
	\end{equation}
\end{lemma}
\begin{proof}
	For convenience, we substitute $U^{\epsilon}$ with $U$ in the proof of this lemma. Taking the inner product of \eqref{6.4} with $U$
	\begin{equation}\label{6.8-1}
	\begin{aligned}
	\frac{1}{2}\frac{d }{d t}\|U\|^2 =&- \nu \|D^{+} U\|^2 -\alpha e^{\epsilon z(\theta_{t}\omega)} \left(U D ^{-}U, U \right) -  \beta e^{2\epsilon z(\theta_{t}\omega)}
	U^4 + \beta(1+\gamma) e^{\epsilon z(\theta_{t}\omega)} U^3\\
	& - (\beta\gamma+\lambda) U^2 + e^{-\epsilon z(\theta_{t}\omega)} 
	\left( f, U \right) + \epsilon U^2 z(\theta_{t}\omega).
	\end{aligned}
	\end{equation}
	By Young's inequality, it is easy to show
	\begin{equation*}
	e^{-\epsilon z(\theta_{t}\omega)} \left(f, U\right) \leq \frac{\lambda-\lambda^*}{2}\|U\|^2 + \frac{1}{2(\lambda-\lambda^*)}e^{-2\epsilon z(\theta_{t}\omega)} \|f\|^2,\hspace{3cm}
	\end{equation*}
	\begin{equation*}
	|-\alpha e^{\epsilon z(\theta_{t}\omega)} \left(U D ^{-}U, U \right)| =\left|-\alpha e^{\epsilon z(\theta_{t}\omega)} \sum_{i\in\mathbb{Z}} \left( U^2_iU_{i-1}-U_i^3 \right)\right| \leq 2 \alpha e^{\epsilon z(\theta_{t}\omega)} \|U\|^3,\hspace{1cm}
	\end{equation*}
	\begin{equation*}
	\left( 2 \alpha+\beta+\beta\gamma \right) e^{\epsilon z(\theta_{t}\omega)} \|U\|^3 \leq \beta e^{2 \epsilon z(\theta_{t}\omega)} \|U\|^4 + \frac{(2\alpha+\beta+\beta\gamma)^2}{4 \beta} \|U\|^2.\hspace{2.5cm}
	\end{equation*}
	From above, we have
	\begin{equation}\label{6.8-5}
	\begin{aligned}
	\frac{1}{2}\frac{d }{d t}\|U\|^2  &\leq  \left[4\nu+\frac{(2\alpha+\beta+\beta\gamma)^2}{4 \beta} - (\beta\gamma+\lambda) + \frac{\lambda-\lambda^*}{2}\right]\|U\|^2 
	+ \frac{1}{2(\lambda-\lambda^*)} e^{-2\epsilon z(\theta_{t}\omega)} \|f\|^2 + \epsilon z(\theta_{t}\omega) U^2\\
	&=\frac{\lambda^*-\lambda}{2}\|U\|^2 
	+ \frac{1}{2(\lambda-\lambda^*)} e^{-2\epsilon z(\theta_{t}\omega)} \|f\|^2 + \epsilon z(\theta_{t}\omega) U^2.
	\end{aligned}
	\end{equation}
	According to the Gronwall lemma on the interval $[\tau-t,r]~(r \geq \tau-t)$, and substituting $\omega$ by $\theta_{-t}\omega$, it follows that
	\begin{equation*}
	\|U_t\|^2 \leq  e^{\int_{-t}^{r-\tau} 2\epsilon z(\theta_{s}\omega)ds-(\lambda-\lambda^*)(r+t-\tau)}\|U_{\tau-t}\|^2
	+ \frac{1}{\lambda-\lambda^*}\int_{-t}^{r-\tau} e^{-2\epsilon z(\theta_{s}\omega) + \int_{s}^{r-\tau} 2\epsilon z(\theta_{r}\omega)dr + (\lambda-\lambda^*)(s + \tau - r)} \|f\|^2 ds.
	\end{equation*}
	Note that
	\begin{equation*}
		\begin{aligned}
	&\int_{-t}^{r-\tau} e^{-2\epsilon z(\theta_{s}\omega) + \int_{s}^{r-\tau} 2\epsilon z(\theta_{r}\omega)dr + (\lambda-\lambda^*)(s + \tau - r)} \|f\|^2 ds\\
	&=  e^{a(\tau-r)} e^{-\int_{r-\tau}^{0} 2\sigma z(\theta_{r}\omega)dr} \int_{-t}^{r-\tau} e^{-2\epsilon z(\theta_{s}\omega) - \int_{0}^{s} 2\epsilon z(\theta_{r}\omega)dr + (\lambda-\lambda^*)s} \|f\|^2 ds.
		\end{aligned}
	\end{equation*}
	By \eqref{6.2}, we find that there exists a $ T_{1}=T_{1}(\tau,\omega,E, \epsilon_0)(0<\epsilon<\epsilon_0)$ such that for any $t\geq T_{1}>0$, we have  
	\begin{equation*}
	|z(\theta_{-t}\omega)| \leq \frac{(\lambda-\lambda^*)t}{8\epsilon_0},~~~~~|\int_{0}^{-t} z(\theta_{r}\omega)dr| \leq \frac{(\lambda-\lambda^*)t}{8\epsilon_0}.
	\end{equation*}
	For any $t\geq T_{1}>0$, then we get 
	\begin{equation*}
	\int_{-t}^{T_1} e^{-2\epsilon z(\theta_{s}\omega) - \int_{0}^{s} 2\epsilon z(\theta_{r}\omega)dr + (\lambda-\lambda^*)s} \|f\|^2 ds
	\leq \int_{-t}^{-T_1} e^{\frac{\lambda-\lambda^*}{2}s} \|f\|^2 ds.
	\end{equation*}
	Since $f\in \ell^2$, we have
	\begin{equation}\label{6.8-2}
	\int_{-t}^{-T_1} e^{\frac{\lambda-\lambda^*}{2}s} \|f\|^2 ds< +\infty.
	\end{equation}
	Then we get
	\begin{equation*}
	\int_{-t}^{r-\tau} e^{-2\epsilon z(\theta_{s}\omega) - \int^{s}_{r-\tau} 2\epsilon z(\theta_{r}\omega)dr + (\lambda-\lambda^*)(s + \tau - r)} \|f\|^2 ds \leq
	\int_{-\infty}^{r-\tau} e^{-2\epsilon z(\theta_{s}\omega) - \int^{s}_{r-\tau} 2\epsilon z(\theta_{r}\omega)dr +(\lambda-\lambda^*)(s + \tau - r)} \|f\|^2 ds,
	\end{equation*}
	the integral is convergent due to \eqref{6.8-2}.
	Therefore, we yield
	\begin{equation*}
	\begin{aligned}
	\|U(r,\tau-t,\theta_{-\tau}\omega,U_{\tau-t})\|^2 \leq&  e^{\int_{-t}^{r-\tau} 2\epsilon z(\theta_{s}\omega)ds-(\lambda-\lambda^*)(r+t-\tau)}\|U_{\tau-t}\|^2\\
&	+\frac{1}{\lambda-\lambda^*} \int_{-\infty}^{r-\tau} e^{-2\epsilon z(\theta_{s}\omega) + \int_{s}^{r-\tau} 2\epsilon z(\theta_{r}\omega)dr + (\lambda-\lambda^*)(s + \tau - r)} \|f\|^2 ds.
\end{aligned}
	\end{equation*}
	In particular, when $r=\tau$, we have
	\begin{equation*}
	\|U(\tau,\tau-t,\theta_{-\tau}\omega,U_{\tau-t})\|^2 \leq  e^{\int_{-t}^{0} 2\epsilon z(\theta_{s}\omega)ds-at}\|U_{\tau-t}\|^2
	+ \frac{1}{\lambda-\lambda^*}\int_{-\infty}^{0} e^{-2\epsilon z(\theta_{s}\omega) + \int_{s}^{0} 2\epsilon z(\theta_{r}\omega)dr + (\lambda-\lambda^*)s} \|f\|^2 ds.
	\end{equation*}
	According to the properties of the Ornstein-Uhlenbeck process, it follows that
	\begin{equation*}
	\int_{-\infty}^{0} e^{-2\epsilon z(\theta_{s}\omega) + \int_{s}^{0} 2\epsilon z(\theta_{r}\omega)dr + (\lambda-\lambda^*)s} ds < +\infty.
	\end{equation*}
	Note that $\{E(\tau,\omega):\tau\in \mathbb{R},\omega\in\Omega\}\in \mathcal{D}$ is tempered, then for any $U_{\tau-t}\in E(\tau-t,\theta_{-t}\omega)$, we have 
	\begin{equation*}
	\lim_{t \rightarrow + \infty}e^{\int_{-t}^{0} (2\epsilon z(\theta_{s}\omega)-(\lambda-\lambda^*))ds}\|U_{\tau-t}\|^2=0.
	\end{equation*}
	We complete the proof.
\end{proof}

According to Lemma \ref{lemma6.1}, the random dynamical system $\Gamma^{\epsilon}(\cdot,\omega)$ has a random $\tilde{\mathcal{D}}$-absorbing set $\mathcal{E}\in\tilde{\mathcal{D}}$,
which is described by
\begin{equation}\label{6.R}
\mathcal{E}^{\epsilon}(\omega) = \left\{ \Phi^{\epsilon} \in X: \parallel \Phi^{\epsilon} \parallel _{X}^{2} \leq R(\epsilon,\omega) \right\},\mspace{6mu}\omega \in \Omega.
\end{equation}

\begin{lemma}\label{lemma6.2}
	For any $\delta>0, \mathcal{D}\in\tilde{\mathcal{D}}$ and $\omega\in \Omega$, there exists an $n> N_{1}$ such that for all $t\geq T_{2}(\delta,\omega,\mathcal{E}^{\epsilon}(\theta_{-t}\omega))$, $\Phi_{0}\in\mathcal{D}(\theta_{-t}\omega)$
	\begin{equation}
	\parallel U\left( t,\theta_{- t}\omega,U_{0} \right) \parallel_{X(|i| \geq I_{1}(\delta,\omega))}^{2}  \leq \delta.
	\end{equation}
\end{lemma}
\begin{proof} 
	For convenience, we also substitute $U^{\epsilon}$ with $U$ during the process of proof. Taking the inner product of \eqref{6.4} with $\xi_{n}U$, where $\xi_n$ is defined as in Theorem \ref{theorem3.2},
	we get
	\begin{equation}\label{6.2-1}
	\begin{aligned}
	\frac{d}{dt}\sum_{i\in\mathbb{Z}}\xi_{n}|U_{i}|^{2}=& 2\nu \langle \Lambda U,\xi_{n}U\rangle-2\alpha e^{\epsilon z(\theta_{t}\omega)} \langle UD^{-}U,\xi_{n}U\rangle - 2\lambda \langle U,\xi_{n}U\rangle - 2\beta e^{2\epsilon z(\theta_{t}\omega)} \langle U^3,\xi_{n}U\rangle \vspace{2.0ex}\\
	&
	+ 2\beta(1+\gamma)e^{\epsilon z(\theta_{t}\omega)} \langle U^2,\xi_{n}U\rangle -  2\beta\gamma e^{\epsilon z(\theta_{t}\omega)} \langle U,\xi_{n}U\rangle + 2e^{-\epsilon z(\theta_{t}\omega)}\langle f,\xi_{n}U\rangle +2\epsilon \langle Uz(\theta_{t}\omega),\xi_{n}U\rangle.
	\end{aligned}
	\end{equation}
	By the Young's inequality, we have
	\begin{equation}
	2e^{-\epsilon z(\theta_{t}\omega)} \langle f,\xi_{n}U\rangle
	\leq (\lambda-\lambda^*)\sum_{i\in \mathbb{Z}}\xi_{n,i}|U_{i}|^{2}
	+\frac{1}{\lambda-\lambda^*} e^{-2\epsilon z(\theta_{t}\omega)}\sum_{i\in \mathbb{Z}}|f_{i}|^{2},
	\end{equation}
	\begin{equation}
	2\alpha e^{\epsilon z(\theta_{t}\omega)} \langle UD^{-}U,\xi_{n}U\rangle  = 2\alpha e^{\epsilon z(\theta_{t}\omega)} \sum_{i\in \mathbb{Z}} \xi_{n,i}(U_{i}^2 U_{i-1}-U_{i}^3)
	\leq 4\alpha e^{\epsilon z(\theta_{t}\omega)}\sum_{i\in \mathbb{Z}} \xi_{n,i}|U_{i}|^{3},
	\end{equation}
	\begin{equation}
	\left(4\alpha + 2\beta + 2\beta\gamma \right) e^{2\epsilon z(\theta_{t}\omega)} \sum_{i\in \mathbb{Z}} \xi_{n,i}|U_{i}|^3  \leq \beta e^{2\epsilon z(\theta_{t}\omega)}\sum_{i\in \mathbb{Z}} \xi_{n,i} |U_{i}|^{4} + \frac{(4\alpha + 2\beta + 2\beta\gamma)^2}{4\beta}\sum_{i\in \mathbb{Z}} \xi_{n,i} |U_{i}|^{2}.
	\end{equation}
	Using Gronwall's inequality on \eqref{6.2-1} from $ T_{k}=T_{k}(\omega)\geq0$ to $t\geq T_{k}$, and then substituting $\omega$ with $\theta_{-t}\omega$, we obtain
	\begin{equation}\label{J}
	\begin{aligned}
	&\hspace{0.5cm}\sum_{i\in \mathbb{Z}}\xi_{n}|U_{i}(t,\theta_{-t}\omega,U_{0}(\theta_{-t}\omega))|^{2} \leq e^{- (\lambda-\lambda^*)(t-T_{k})+\int_{T_{k}}^{t}2\epsilon z(\theta_{s-t}\omega)ds}
	\|U(T_{k},\theta_{-t}\omega,U_{0}(\theta_{-t}\omega))\|^{2}   \\
	&\hspace{5cm}
	+\frac{1}{\lambda-\lambda^*}\int_{T_{k}}^{t}e^{-(\lambda-\lambda^*)(t-\tau)
		+\int_{\tau}^{t}2\epsilon z(\theta_{s-t}\omega)ds}\cdot e^{-2\epsilon z(\theta_{\tau-t}\omega)}d\tau \sum_{i\in \mathbb{Z}}|f_{i}|^{2}=J_{1}+J_{2}.
	\end{aligned}
	\end{equation}
	By \eqref{6.R}, we can get $J_1 \rightarrow 0$ as $t \rightarrow +\infty$. This indicates that for any $\delta>0$, there exists a $T_{1}=T_{1}(\delta,\omega,\mathcal{E}^{\epsilon}(\theta_{-t}\omega))\geq T_{k}$ such that
	\begin{equation}\label{J.1}
	\begin{aligned}
	J_{1}(t) \leq\frac{\delta}{2}e^{-2z(\omega)}, ~~~t\geq T_{1}(\delta,\omega,\mathcal{E}^{\epsilon}).
	\end{aligned}
	\end{equation}
	According to 
	\begin{equation*}
	\begin{aligned}
	\int_{T_{k}}^{t}e^{- (\lambda-\lambda^*) (t-\tau)+\int_{\tau}^{t}2\epsilon z(\theta_{s-t}\omega)ds}\cdot e^{-2\epsilon z(\theta_{\tau-t}\omega)}d\tau
	\leq \int_{-\infty}^{0}e^{ (\lambda-\lambda^*) \tau-\int_{\tau}^{0}2\epsilon z(\theta_{s}\omega)ds}\cdot e^{-2\epsilon z(\theta_{\tau}\omega)}d\tau<\infty,
	\end{aligned}
	\end{equation*}
	and $f\in\ell^{2}$, there exists $I_{1}=I_{1}(\delta,\omega)\in \mathbb{N}$ such that
	\begin{equation}\label{J.2}
	\begin{aligned}
	J_{2}(t) \leq\frac{\delta}{2}e^{-2z(\omega)}, ~~~i> I_{1}.
	\end{aligned}
	\end{equation}
	In conclusion, by substituting \eqref{J.1} and \eqref{J.2} into \eqref{J}, we derive that
	\begin{equation}\label{J.3}
	\sum_{i\geq I_{1}}
	\xi_{n} |U_{i}(t,\theta_{-t}\omega,U_{0}(\theta_{-t}\omega))|^{2}
	\leq \delta e^{-2z(\omega)},~~~t> T_{1},~~i\geq I_{1}.
	\end{equation}
	Based on \eqref{J.3}, we can deduce that
	\begin{equation}\label{J.4}
	\sum_{i\geq I_{1}}\xi_{n} |u_{i}(t,\theta_{-t}\omega,u_{0}(\theta_{-t}\omega))|^{2} \leq \delta.
	\end{equation}
	This proof is completed.
\end{proof}

Note that a random compact set $\mathcal{A}(\omega)$ of $X$ is called a random $\tilde{\mathcal{D}}$-attractor for a random dynamical system $\Phi$ with semigroup $\mathcal{G}(\cdot)$ if $\mathcal{A}:=\{\mathcal{A}(\omega)\}\in\tilde{\mathcal{D}}$, and 
$\mathcal{A}$ satisfies the invariance property, meaning that for any $t\geq0$ and $\omega\in\Omega$, $\Phi(t,\omega)\mathcal{A}(\omega)=\mathcal{A}(\theta_{t}\omega)$. Additionally, $\mathcal{A}$ is $\tilde{\mathcal{D}}$-attracting, i.e., 
\begin{equation}
\lim\limits_{t\rightarrow\infty}\text{dist}_{X}\left( \Phi\left( t,\theta_{- t}\omega \right)\mathcal{D}\left( \theta_{- t}\omega \right),\mathcal{A}(\omega) \right) = 0,\mspace{6mu}\forall\mathcal{D} \in \tilde{\mathcal{D}},\mspace{6mu}\omega \in \Omega.
\end{equation}

The subsequent theorem establishes the existence of the random attractor, as well as the upper semi-convergence relationship between the random attractor and the global attractor.

\begin{theorem}\label{theorem6.3} The random dynamical system $\Phi^{\epsilon}$ generated by the problem \eqref{6.4} has a random attractor $\mathcal{A}(\omega)$ in $X$. Furthermore, $\mathcal{A}(\omega)$ is upper semi-convergence to the global attractor
	$\mathcal{A}$, which means that
	\begin{equation}\label{6.11}
	\lim\limits_{\epsilon\rightarrow 0}\text{dist}_{X}\left( \mathcal{A}^{\epsilon}(\omega),\mathcal{A} \right) = 0,\mspace{6mu} \mathbb{P}\text{-a.s.}\omega \in \Omega.
	\end{equation}
\end{theorem}
\begin{proof}
	By Lemma \ref{lemma6.2}, the equations \eqref{1.1}  possess a unique global random $\tilde{\mathcal{D}}$ attractor given by
	\begin{equation*}
	\mathcal{A}(\omega)=\bigcap\limits_{\tau\geq T_{k}}\overline{\bigcup\limits_{t\geq\tau}\phi(t,\theta_{-t}\omega,\mathcal{E}(\theta_{-t}\omega))}\in X,~~\tau\in \mathbb{R},~~\omega\in\Omega.
	\end{equation*}
	Since the fact that the radius
	$R(\epsilon,\omega)$ in \eqref{6.8} is bounded as $\epsilon\rightarrow0$ and the tail-estimate in Lemma \ref{lemma6.2} is uniform for all $\epsilon\in(0,\epsilon^{*}]$, the uniform asymptotic compactness can be ensured.
	Moreover, by employing a method similar to that in Lemma \ref{lemma6.4} which will be presented later, we can establish the convergence of the random system $\Gamma^{\epsilon}(\cdot,\omega)$ to the semigroup $\mathcal{G}(\cdot)$ as $\epsilon\rightarrow0$. In summary, the upper semi-continuity \eqref{6.11} can be derived by applying the abstract result from \cite{Wang-2014}.
\end{proof}

\subsection{Convergence between random and deterministic truncated attractors}
Using the notations of the two matrices $\Lambda_m$ and $D_{m}^{-}$ introduced in the previous section, we consider the truncation of the random lattice system \eqref{6.4} in the space $X_{m}$ as follows:
\begin{equation}\label{6.14}
\begin{aligned}
\frac{dU^{\epsilon,m}}{dt}= & \nu \Lambda_{m} U^{\epsilon,m}-\alpha e^{\epsilon z(\theta_{t}\omega)} U^{\epsilon,m} D_{m}^{-} U^{\epsilon,m} - \beta e^{2\epsilon z(\theta_{t}\omega)} \left(U^{\epsilon,m}\right)^{3}
\vspace{2.0ex}\\
& +\beta(1+\gamma) e^{\epsilon z(\theta_{t}\omega)} \left(U^{\epsilon,m}\right)^{2} - (\beta \gamma+\lambda) U^{\epsilon,m} + e^{-\epsilon z(\theta_{t}\omega)}f^m+\epsilon U^{\epsilon,m}z(\theta_{t}\omega),
\end{aligned}
\end{equation}
with the initial condition $U^{\epsilon,m}(0)=U_0^\epsilon\in X_m$.
Then, for the truncated random system \eqref{6.14}, there exists a unique solution. This solution generates a continuous random dynamical system
$\Xi_m^\epsilon:\mathbb R^+\times \Omega\times X_m \to X_m$ corresponding to $\epsilon\in(0,\epsilon^*],m\in \mathbb N$, which is denoted as
\begin{equation*}
\Xi_m^\epsilon(t,\omega)U_0^\epsilon=\Phi^{\epsilon,m}(t,\omega,U_0^\epsilon),\forall t\geq 0,\omega\in\Omega,U_0^\epsilon\in X_m.
\end{equation*}
Set $\tilde{\mathcal{D}}_{m}$ be the restriction of $\tilde{\mathcal{D}}$ on $X_{m}$, that is, $\mathcal{D}\in\tilde{\mathcal{D}}_{m}$ if and only if
\begin{equation*}
\lim\limits_{t\rightarrow +\infty}e^{-\rho t}\sup\{\| \Phi^{\epsilon}\|_{X_{m}} :\Phi^{\epsilon}\in \mathcal{D}\left( \theta_{- t}\omega \right)\} = 0,\mspace{6mu} \omega\in\Omega,~~ \text{for all~}\rho>0.
\end{equation*}

For deriving the upper semi-convergence between the truncated semigroup and the global attractor, we proceed to demonstrate the following
$\tilde{\mathcal{D}}_{m}$-absorbing set and tail estimates for the truncated random system. These are correspond to Lemmas \ref{lemma6.1} and \ref{lemma6.2}.

\begin{lemma}\label{lemma6.5}
	Let $\epsilon\in(0,\epsilon^{*}],m\in \mathbb{N}$. Then the random dynamical system $
	\Xi_{m}^{\epsilon}$
	has a random $\tilde{\mathcal{D}}_{m}$-absorbing set given by
	\begin{equation}
	\mathcal{E}_{m}^{\epsilon}(\omega) = \left\{ \Phi_{ m}^{\epsilon} \in X_{m}: \parallel \Phi \parallel \right._{X_{m}}^{2}{\left. \leq R(\epsilon,\omega) \right\}},\mspace{6mu}\omega \in \Omega,
	\end{equation}
	where $R(\epsilon,\omega)$ is given by \eqref{6.8} and independent of $m$.
	Moreover, for each $\delta>0,\mathcal{D}_{m}\in\tilde{\mathcal{D}}_{m}$ and $\omega\in\Omega$, there are $T(\delta)>0$ and $I(\delta)\in \mathbb{N}$ such that, for all $t\geq T,m>I(\delta)$ and $
	\phi_{0}^{m} \in \mathcal{D}_{m}\left( \theta_{- t}\omega \right)$,
	\begin{equation*}
	|\Phi^{\epsilon,m}(t,\theta_{-t}\omega,\Phi_{0}^{m})| \leq\sum_{I\leq|i|\leq m}|u_{i}^{\epsilon,m}|^{2}
	\leq\delta.
	\end{equation*}
\end{lemma}

\begin{theorem}\label{theorem6.2}
	Assume $\epsilon\in(0,\epsilon^{*}]$. For each $m\in\mathbb{N}$, the random dynamical system $\Xi_{m}^{\epsilon}$ has a random
	$\tilde{\mathcal{D}}$-attractor $\mathcal{A}_{m}^{\epsilon}(\omega) \subset X_{m}$ with the following upper semi-continuity:
	\begin{equation}\label{6.18}
	\lim\limits_{m\rightarrow\infty}\text{dist}_{X}\left( \mathcal{A}_{m}^{\epsilon}(\omega),\mathcal{A}^{\epsilon}(\omega) \right) = 0,\mspace{6mu}\forall\omega \in \Omega,
	\end{equation}
	where $\mathcal{A}_{m}^{\epsilon}(\omega)$ is  naturally embedded into $X$ and $\mathcal{A}^{\epsilon}(\omega)$  is the random attractor given in Theorem \ref{theorem6.3}.
\end{theorem}
\begin{proof}  By Lemma \ref{lemma6.2}, the semi-dynamical system $\phi$ is asymptotically compact in $\tilde{\mathcal{D}}_{m}$, then the existence of the random attractor holds. Similar to the case of stochastic PDE on expanding domains as given in \cite{LiF-2019,LiY-2019}, we can obtain upper semi-continuity \eqref{6.18}.
	In addition, by the similar method as in Lemma \ref{lemma6.4} later, we can prove the convergence $\Gamma^{\epsilon}_{m}\rightarrow\Gamma^{\epsilon}$
	as $m\rightarrow\infty$.
	Moreover, we can verify the absorption and tail-estimate in Lemma \ref{lemma6.5} are uniform for enough large $m$, thus we can prove that the asymptotic
	compactness of $\Gamma^{\epsilon}_{m}(t,\omega)$ is uniform if $m$ is large enough .
\end{proof}

The deterministic lattice system \eqref{2.1} on the finite-dimensional space $X_{m}$ can be written as
\begin{equation}\label{6.15}
\frac{du^{m}}{dt}=  \nu \Lambda_{m} u^{m}-\alpha u^{m} D_{m}^{-} u^{m} + \beta  u^{m} \left(1-u^{m}\right)\left(u^{m}-\gamma\right) -\lambda u^{m} + f^m.
\end{equation}
For each $m\in \mathbb{N}$, there exists a unique solution of \eqref{6.15} and which generates a semigroup
$ \mathcal{G}_{m}:\mathbb{R}^{+} \times X_{m}\rightarrow X_{m}$ defined by
\begin{equation}\label{6.20}
\mathcal{G}_{m}(t)u_{0}^{m} = u^{m}\left( t,u_{0}^m \right),\mspace{6mu}
\forall t \geq 0,\mspace{6mu}u_{0}^{m} \in X_{m}.
\end{equation}
In order to prove the truncated semigroup $\mathcal{G}_{m}(\cdot)$ has a global attractor $\mathcal{A}_{m}$ in $X_{m}$, which also converges to $\mathcal{A}_{m}^{\epsilon}(\omega)$ as $\epsilon\rightarrow0$. We need show that, for each $m\in \mathbb{N}$, the semigroup $\mathcal{G}_{m}(\cdot)$ is the limit of the truncated random dynamical system $
\Xi_{m}^{\epsilon}(\cdot ,\omega)$ as $\epsilon\rightarrow0$.

\begin{lemma}\label{lemma6.4}
	For any $m\in \mathbb{N}$, when $|\Phi_{0}^{\epsilon}-u_{0}^m|_{X_{m}}\rightarrow0$ as $\epsilon\rightarrow0$, the solutions of \eqref{6.14} and \eqref{6.15} satisfy the following:
	\begin{equation}\label{lemma6.7}
	\lim\limits_{\epsilon\rightarrow 0}\left| \Phi^{\epsilon,m}\left( T,\theta_{- T}\omega,\Phi_{0}^{\epsilon} \right) - u^{m}\left( T,u_{0}^m \right) \right|_{X_{m}} = 0,\mspace{6mu}\forall T > 0,\mspace{6mu}\omega \in \Omega.
	\end{equation}
\end{lemma}
\begin{proof}
	For simplicity, we drop the superscript $m$. Let $u^{\epsilon}(t,\omega) = U^{\epsilon}\left( t,\omega,u_{0}^{\epsilon} \right)-u\left( t,u_{0} \right),\mspace{6mu}\forall t \geq 0.$
	Then, by subtracting \eqref{6.15} from \eqref{6.14}, one has
	\begin{equation}\label{6.22}
	\begin{aligned}
	\frac{du^{\epsilon}}{dt}=&\nu \Lambda_{m} u^{\epsilon} - \alpha \left( e^{\epsilon z(\theta_{t}\omega)} U^{\epsilon} D_{m}^{-} U^{\epsilon} - u D_{m}^{-} u \right) - \beta \left[ e^{2\epsilon z(\theta_{t}\omega)} \left(U^{\epsilon}\right)^{3} - u^{3}\right]
	\vspace{2.0ex}\\
	& +\beta(1+\gamma) \left[ e^{\epsilon z(\theta_{t}\omega)} \left(U^{\epsilon}\right)^{2} - u^2\right] - (\beta \gamma+\lambda) \left[ U^{\epsilon}-u\right] + \left( e^{-\epsilon z(\theta_{t}\omega)}-1 \right) f + \epsilon z(\theta_{t}\omega) U^{\epsilon}.
	\end{aligned}
	\end{equation}
	By taking the inner product of \eqref{6.22} with $u^{\epsilon}$ in $X_{m}$, we obtain
	\begin{equation}\label{6.23}
	\begin{aligned}
	\frac{d|u^{\epsilon}|^{2}}{dt} + 2\nu |D^{+}_{m}u^{\epsilon}|^{2}+2(\beta\gamma+\lambda)|u^{\epsilon}|^{2}
	= & -2\alpha \left(e^{\epsilon z(\theta_{t}\omega)} U^{\epsilon} D^{-}_{m} U^{\epsilon}-uD^{-}_{m}u\right)\cdot u^{\epsilon}\\
	& -2\beta\left[e^{2\epsilon z(\theta_{t}\omega)}\left(U^{\epsilon}\right)^{3}-u^{3}\right]\cdot u^{\epsilon} + 2\beta(1+\gamma)\left[e^{\epsilon z(\theta_{t}\omega)}\left(U^{\epsilon}\right)^{2}-u^2 \right] \cdot u^{\epsilon}\\
	& +2\left(e^{-\epsilon z(\theta_{t}\omega)}-1 \right) f^{m}\cdot u^{\epsilon}
	+2\epsilon z(\theta_{t}\omega) U^{\epsilon} \cdot u^{\epsilon}.
	\end{aligned}
	\end{equation}
	Recall that $U^{\epsilon}=e^{-\epsilon z(\theta_{t}\omega)}u$, which implies that
	\begin{equation}\label{6.26}
	\begin{aligned}
	-2\beta\left[e^{2\epsilon z(\theta_{t}\omega)} \left( U^{\epsilon} \right)^{3}-u^{3}\right] 
	& =-2\beta\left[e^{2\epsilon z(\theta_{t}\omega)}\left( U^{\epsilon} \right)^{3} - e^{3\epsilon z(\theta_{t}\omega)} \left( U^{\epsilon} \right)^{3} \right]      \\
	& =-2\beta e^{2\epsilon z(\theta_{t}\omega)}\left( U^{\epsilon} \right)^{3} \left(1 - e^{\epsilon z(\theta_{t}\omega)} \right),
	\end{aligned}
	\end{equation}
	
	\begin{equation}\label{6.26-1}
	2\beta(1+\gamma)\left[e^{\epsilon z(\theta_{t}\omega)}\left(U^{\epsilon}\right)^{2}-u^2 \right]= 2\beta(1+\gamma)e^{\epsilon z(\theta_{t}\omega)}\left(U^{\epsilon}\right)^{2} \left(1- e^{\epsilon z(\theta_{t}\omega)}\right),
	\end{equation}
	
	\begin{equation}\label{6.26-2}
	\begin{aligned}
	2\alpha \left(e^{\epsilon z(\theta_{t}\omega)} U^{\epsilon} D^{-}_{m} U^{\epsilon}-uD^{-}_{m}u\right)
	& = 2\alpha \left(e^{\epsilon z(\theta_{t}\omega)} U^{\epsilon} D^{-}_{m} U^{\epsilon}-e^{2\epsilon z(\theta_{t}\omega)} U^{\epsilon} D^{-}_{m} U^{\epsilon}\right)\\
	& = 2\alpha e^{\epsilon z(\theta_{t}\omega)} U^{\epsilon} D^{-}_{m} U^{\epsilon}\left(1- e^{\epsilon z(\theta_{t}\omega)} \right).
	\end{aligned}
	\end{equation}
	By Young's inequality, there holds
	\begin{equation}\label{6.27-0}
	2\left(e^{-\epsilon z(\theta_{t}\omega)}-1 \right) f\cdot u^{\epsilon} \leq \left(\beta\gamma+\lambda\right) |{u}^{\epsilon}|^{2}+ \frac{\left(e^{-\epsilon z(\theta_{t}\omega)}-1 \right)^2}{\beta\gamma+\lambda}|f|^{2}.
	\end{equation}
	Let $\tau\in \mathbb R, \omega\in \Omega, T>0$, and $\zeta\in[0,1)$, because $\omega$ is continuous on $\mathbb R$, it follows that there exists $\epsilon_{1}=\epsilon_{1}(\omega,T,\zeta)>0$ such that for each $\epsilon\in(0,\min\{\epsilon_{1},\epsilon^{*}\})$ and $t\in[\tau,\tau+T]$, we have
	\begin{equation}\label{6.21}
	|e^{-\epsilon z(\theta_{t}\omega)}-1|+|e^{\epsilon z(\theta_{t}\omega)}-1|<\zeta.
	\end{equation}
	As a result, we get
	\begin{equation}\label{6.27}
	\begin{aligned}
	&\frac{d(|u^{\epsilon}|^{2})}{dt} + \left(\beta\gamma+\lambda\right)|u^{\epsilon}|^{2} +
	+2\nu |D^{+}_{m}u^{\epsilon}|^{2}\\
	\leq & 3\zeta \left[ \beta^2 e^{4\epsilon z(\theta_{t}\omega)} |U^{\epsilon}|^{6} + \beta^2 \left( 1+\gamma \right)^{2} e^{2\epsilon z(\theta_{t}\omega)} |U^{\epsilon}|^{4} +\alpha^{2} e^{2\epsilon z(\theta_{t}\omega)}|U^{\epsilon}|^{2} |D^{-}_{m} U^{\epsilon}|^{2}\right]\\
	&+\zeta |u^{\epsilon}|^{2} + \frac{\zeta^2}{\beta\gamma+\lambda}|f|^{2} 
	+\epsilon z \left( |U^{\epsilon}|^2+|u^{\epsilon}|^2\right).
	\end{aligned}
	\end{equation}
	We apply Gronwall's inequality to \eqref{6.27} from 0 to $T$ with $\omega$ replaced by $\theta_{-t}\omega$ yields
	\begin{equation}
	\begin{aligned}
	|u^{\epsilon}(T,(\theta_{-T}\omega))|^{2}
	\leq&\exp\left\{-(\beta\gamma + \lambda - \zeta)T+\epsilon\int_{-T}^{0}z(\theta_{s}\omega)ds\right\} |
	u_{0}^{\epsilon}|^{2} \\
	&+\int_{-T}^{0}\exp\left\{-(\beta\gamma + \lambda - \zeta)s+\epsilon\int_{s}^{0}z(\theta_{r}\omega)dr\right\}
	\times \\
	& \times\left[3\zeta e^{2\epsilon z(\theta_{t}\omega)} \left( \beta^{2} e^{2\epsilon z(\theta_{t}\omega)} |U^{\epsilon}(s+T)|^{6}  + \beta^{2} \left( 1+\gamma \right)^{2}  |U^{\epsilon}(s+T)|^{4} \right.\right.\\
	& \left.\left. + \alpha^{2}|U^{\epsilon}(s+T)|^{2} \cdot |D^{-}_m U^{\epsilon}(s+T)|^{2} \right) + \frac{\zeta^2}{\beta\gamma + \lambda} \|f\|^{2} + \epsilon z(\theta_{s}\omega) |U^{\epsilon}(s+T)|^{2} \right]ds\\
	:=&Z_{1}+Z_{2}.
	\end{aligned}
	\end{equation}
	
	Observe that by combining the convergence of the initial data with the boundedness of 
	$z(\theta_{s}\omega)$ in $s\in[-T,0]$, we conclude that
	\begin{equation*}
	\begin{aligned}
	&\exp\left\{-(\beta\gamma + \lambda -  \zeta)T+\epsilon\int_{-T}^{0}z(\theta_{s}\omega)ds\right\}|u_{0}^{\epsilon}|^{2}\\
	\leq& 
	\exp\left\{-(\beta\gamma+\lambda-\zeta)T+\epsilon T\sup\limits_{s\in[-T,0]}z(\theta_{s}\omega)ds\right\}
	|u_{0}^{\epsilon}|^{2}\rightarrow0,
	~~\text{as}~~\epsilon\rightarrow0,
	\end{aligned}
	\end{equation*}
	it is obviously that $\zeta$ is closely related to $\epsilon$ by \eqref{6.21}, and $\zeta\rightarrow0$ as $\epsilon\rightarrow0$.
	Therefore, by the boundedness of $U^{\epsilon}$ in \eqref{6.R} and $f\in\ell^{2}$, when $\epsilon\rightarrow0$, $Z_{2}\rightarrow0$. In summary, one can deduce that
	\begin{equation}
	\lim\limits_{\epsilon\rightarrow 0}\left| \Phi^{\epsilon,m}\left( T,\theta_{- T}\omega,\Phi_{0}^{\epsilon} \right) - u^{m}\left( T,u_{0}^m \right) \right| = 0,\mspace{6mu}\forall T > 0,\mspace{6mu}\omega \in \Omega.
	\end{equation}
	We complete the proof.
\end{proof}

\begin{theorem}\label{theorem6.4} Assume $\epsilon\in(0,\epsilon^{*}]$, the truncated semigroup $\mathcal{G}_{m}(\cdot)$, generated from \eqref{6.15}, has a global attractor $\mathcal{A}_{m}$ in $X_{m}$, which converges to the attractor $\mathcal{A}_{m}^{\epsilon}(\omega)$ as $\epsilon\rightarrow0$, that is
	\begin{equation}\label{6.32}
	\lim\limits_{\epsilon\rightarrow 0}\text{dist}_{X_{m}}\left( \mathcal{A}_{m}^{\epsilon}(\omega),\mathcal{A}_{m} \right) = 0,\mspace{6mu}\forall\omega \in \Omega.
	\end{equation}
\end{theorem}
\begin{proof}
	Note that the fact that $|f^{m}|\leq \|f\|$,
	then applying the methodology developed in Lemma \ref{lemma2.1}, we establish that the truncated semigroup $\mathcal{G}_{m}(\cdot)$ admits a bounded absorbing set, specifically given by
	\begin{equation*}
	\mathcal{E}_{m}:=\left\{u^m \in X_m:\|u^m\|_{X_{m}} \leq r^*:=1+\frac{\|f\|}{\lambda-\lambda^*}\right\}
	=\mathcal{C}_{X_{m}}(0,r^*).
	\end{equation*}
 Since $\mathcal{E}_{m}$ is compact in $X_{m}$, it can be obtained that $\mathcal{G}_{m}(\cdot)$ has a global attractor $\mathcal{A}_{m}$ in $X_{m}$.
	
	To achieve the result \eqref{6.32}, it suffices to verify three essential conditions of the abstract theorem(see \cite{Bates-2009,Li-2015,Wang-2014}).

	\textbf{Convergence of cocycles.} By Lemma \ref{lemma6.4}, when $|U_{0}^{\epsilon} - u_{0}^m|_{X_{m}}\rightarrow 0$ as $\epsilon\rightarrow0$, for each $t>0$ and $\omega\in \Omega$, it follows that
	\begin{equation*}
	\lim\limits_{\epsilon\rightarrow 0}\left| \Xi_{m}^{\epsilon}\left( t,\theta_{- t}\omega,U_{0}^{\epsilon} \right) - \mathcal{G}_{m}(t)u_{0}^m \right|_{X_{m}} = 0.
	\end{equation*}
	
	\textbf{Convergence of absorbing radii.} According to the Lebesgue dominated convergence theorem, one can deduced that
	\begin{equation*}
	\begin{array}{cl}
	{\lim\limits_{\epsilon\rightarrow 0}R(\epsilon,\omega)} 
	& {= 1 + \frac{\|f\|^{2}}{\lambda-\lambda^*} \int_{-\infty}^{0}e^{-2\epsilon z(\theta_{s}\omega) -\int_{0}^{s} 2\epsilon z(\theta_{r}\omega)dr
			+(\lambda-\lambda^*) s} ds} \\
	& {= 1 + \frac{\|f\|^{2}}{(\lambda-\lambda^*)^2} \leq {r^{*}}^{2}.}
	\end{array}
	\end{equation*}

	\textbf{Uniform compactness of attractors.} Observe that since 
	$\mathcal{A}_{m}^{\epsilon}(\omega) \subseteq \mathcal{E}_{m}^{\epsilon}(\omega)$ holds for all $\epsilon\in(0,\epsilon^{*}]$, it follows that for all $\omega\in \Omega$, the absorbing radius $R(\epsilon,\omega)$ is increasing with respect to $\epsilon\in(0,\epsilon^{*}]$, we have
	\begin{equation}
	\bigcup\limits_{0 < \epsilon \leq \epsilon^{*}}\mathcal{A}_{m}^{\epsilon}(\omega) \subseteq \bigcup\limits_{0 < \epsilon \leq \epsilon^{*}}\mathcal{E}_{m}^{\epsilon}(\omega) \subseteq \mathcal{E}_{m}^{\epsilon^{*}}(\omega).
	\end{equation}
	Therefore, we conclude that the union $
	\cup \left\{ \mathcal{A}_{m}^{\epsilon}(\omega):\mspace{6mu} 0 < \epsilon \leq \epsilon^{*} \right\}$ forms a precompact set in the finite-dimensional space $X_{m}$. This establishes the third condition and we complete the proof.
	
\end{proof}

\subsection*{Conflict of interest}
The authors have no conflicts to disclose.

\subsection*{Availability of date and materials}
Not applicable.

\end{document}